\documentclass[12pt]{article}
\usepackage{amsmath}
\usepackage{amsthm}
\usepackage{amssymb}
\usepackage{tikz}
\usepackage[colorlinks=true, linkcolor=blue]{hyperref}
\usepackage[shortlabels]{enumitem}

\DeclareRobustCommand{\stirling2}{\genfrac\{\}{0pt}{}}
\DeclareRobustCommand{\eulerian}{\genfrac\langle\rangle{0pt}{}}

\DeclareMathOperator{\tcomp}{t-comp}
\DeclareMathOperator{\op2}{op2}
\DeclareMathOperator{\qt}{qt}
\DeclareMathOperator{\qtconn}{qt-conn}
\DeclareMathOperator{\qtcomp}{qt-comp}
\DeclareMathOperator{\lt}{lt}
\DeclareMathOperator{\ltcomp}{lt-comp}
\DeclareMathOperator{\ld}{ld}
\DeclareMathOperator{\qlt}{qlt}
\DeclareMathOperator{\qltconn}{qlt-conn}
\DeclareMathOperator{\qltcomp}{qlt-comp}

\theoremstyle{plain}
\newtheorem{theorem}{Theorem}

\newtheorem{claim}[theorem]{Claim}

\theoremstyle{definition}
\newtheorem{definition}[theorem]{Definition}
\newtheorem{example}[theorem]{Example}

\theoremstyle{remark}

\newcommand{\seqnum}[1]{\href{https://oeis.org/#1}{\rm \underline{#1}}}

\usepackage{url}

\begin{document}

\title{Enumerating Threshold Graphs and Some Related Graph Classes}

\author{David Galvin, Greyson Wesley and Bailee Zacovic\thanks{Department of Mathematics,
University of Notre Dame, Notre Dame IN; {\tt dgalvin1}, {\tt gwesley} and {\tt bzacovic@nd.edu}.}}

\maketitle

\begin{abstract}
We give combinatorial proofs of some enumeration formulas involving labelled threshold, quasi-threshold, loop-threshold and quasi-loop-threshold graphs. In each case we count by number of vertices and number of components. For threshold graphs, we also count by number of dominating vertices, and for loop-threshold graphs we count by number of looped dominating vertices. 

We also obtain an analog of the Frobenius formula (connecting Eulerian numbers and Stirling numbers of the second kind) in the context of labelled threshold graphs.
\end{abstract}

\section{Introduction and summary of results}

A {\it threshold graph} is a graph $G$ for which there is a real number $t$ (the threshold), and an assignment $a:V(G)\rightarrow {\mathbb R}$ of real numbers to the vertices of $G$, with the property that $uv$ is an edge of $G$ if and only if $a(u)+a(v)> t$. Equivalently, there is $s \in {\mathbb R}$ and $b:V(G)\rightarrow {\mathbb R}$ such that $I \subseteq V(G)$ is an independent set (a set spanning no edges) if and only if $\sum_{v \in I} b(v) \leq s$.

Threshold graphs were introduced by Chv\'atal \& Hammer \cite{ChvatalHammer1} in relation to a set packing problem, and were later independently discovered by various other authors in mathematics, operations research, computer science and psychology. The wide applicability of threshold graphs is due in part to the fact that they admit many alternate characterizations. Indeed, the monograph on threshold graphs by Mahadev and Peled begins by establishing eight distinct characterizations \cite[Theorem 1.2.4]{MahadevPeled} (and also gives a good history of their appearance in various fields).   

The most germane characterization for the purposes of this note is that the family of threshold graphs is
the smallest (by containment) family of (simple, finite) graphs that contains $K_1$, and is closed under adding isolated vertices and adding dominating vertices (vertices adjacent to all other vertices in the graph).

Let $t_n$ denote the number of labelled threshold graphs on $n$ vertices (almost always in this note, we will implicitly take the label set to be $[n]:=\{1,\ldots, n\}$). The beginning of the sequence $(t_n)_{n \geq 1}$ is shown in Table \ref{table-tn}, and appears in the \href{https://oeis.org}{On-Line Encyclopedia of Integer Sequences} ({\it OEIS}) \cite{OEIS}, as entry \seqnum{A005840}.

\begin{table}[ht!]
\begin{center}
\begin{tabular}{r|rrrrrrc}
$n$ & 1 & 2 & 3 & 4 & 5 & 6 & $\cdots$ \\
\hline
$t_n$ & 1 & 2 & 8 & 46 & 332 & 2874 & $\cdots$ 
\end{tabular}
\caption{The count of labelled threshold graphs on vertex set $[n]$.} \label{table-tn}
\end{center}
\end{table}

An {\it ascent} in a permutation $\pi=\pi_1\cdots\pi_n$ of $[n]$ (given here in one-line notation, as all permutations in this note will be) is an index $i \in [n-1]$ with $\pi_i<\pi_{i+1}$. The Eulerian number $\eulerian{n}{k}$ counts the number of permutations of $[n]$ with $k$ ascents (OEIS \seqnum{A008292}; but note that that entry uses an offset indexing, counting permutations with $k-1$ ascents). Beissinger and Peled \cite{BeissingerPeled} implicitly observed, via a generating function argument, the formula
\begin{equation} \label{Spiro}
t_n = \sum_{k=1}^{n-1}(n-k)\eulerian{n-1}{k-1}2^k,
\end{equation}
valid for $n \geq 2$. Spiro \cite{Spiro} recently gave a combinatorial proof of \eqref{Spiro}. In this note, we give combinatorial proofs of other enumeration formulas for threshold graphs, and for some related graph families --- quasi-threshold graphs, loop-threshold graphs, and quasi-loop-threshold graphs. 

In Sections \ref{sec-ltgr} (threshold graphs), \ref{sec-lqtgr} (quasi-threshold graphs), \ref{sec-lt-intro} (loop-threshold graphs) and \ref{sec-qlt-results} (quasi-loop-threshold graphs) we present and discuss our results, and then we give the proofs in Sections \ref{sec-thresh-proofs}, \ref{sec-quasi-thresh-proofs}, \ref{sec-loop-thresh-proofs} and \ref{sec-quasi-loop-thresh-proofs}.

\subsection{Labelled threshold graph results} \label{sec-ltgr}

We begin in Section \ref{subsec-spiros-proof} by recalling Spiro's combinatorial proof of \eqref{Spiro}, as it introduces ideas and notation that will be useful throughout the note.

Next we present combinatorial proofs of two other formulas for $t_n$ that are presented at OEIS \seqnum{A005840}: for $n \geq 1$, and taking $t_0=1$,
\begin{equation} \label{tn1}
t_n = 1-n +\sum_{k=0}^{n-1} \binom{n}{k}t_k
\end{equation}
(Section \ref{subsec-tn1-proof}), and for $n \geq 1$, and taking $t_0=t_1=1$,
\begin{equation} \label{tn2}
t_{n+1} = n(t_n-t_{n-1}) +\sum_{k=0}^n \binom{n}{k}t_kt_{n-k}.
\end{equation}
(Section \ref{subsec-tn2-proof}). 

In Section \ref{sec-thresh-comp-proofs}, 
we consider the enumeration of labelled threshold graphs by number of components. This work has already been carried out, using generating functions, by Beissinger and Peled \cite{BeissingerPeled} (and in fact they added another parameter, number of distinct vertex degrees). Here, we take a fully combinatorial approach. For $n, k \geq 1$, let $\tcomp_{n,k}$ denote the number of labelled threshold graphs on vertex set $[n]$ with $k$ components. We establish that
\begin{enumerate}[(i)]
\item $\tcomp_{1,1} = 1$,
\item for $n \geq 2$,  $\tcomp_{n,1} = \sum_{k=1}^{n-1} (n-k)\eulerian{n-1}{k-1} 2^{k-1}~\left(= \frac{t_n}{2}\right)$,
\item for $n \geq 3$ and $2 \leq k \leq n-1$, $\tcomp_{n,k} = \binom{n}{k-1} \tcomp_{n-k+1,1}$ and
\item for $n \geq 2$, $\tcomp_{n,n}=1$.
\end{enumerate}
(These all follow in a straightforward way from \eqref{Spiro}.)

The triangle $(\tcomp_{n,k}:n \geq 1, 1 \leq k \leq n)$ is shown in Table \ref{table-tcomp}, and is OEIS \seqnum{A348436}. The first column of the triangle is the sequence enumerating connected labelled threshold graphs on $n$ vertices ($n \geq 1$), and is OEIS \seqnum{A317057}. (A different sequence, OEIS \seqnum{A053525}, is listed as enumerating connected labelled threshold graphs on $n$ vertices. This sequence is identical to OEIS \seqnum{A317057} for $n \geq 2$, but at $n=1$ takes the value $0$ rather than the correct value of $1$. Also, at OEIS \seqnum{A005840} there is an incorrect conjecture concerning the number of connected labelled threshold graphs.) 

\begin{table}[ht!]
\begin{center}
\begin{tabular}{r|rrrrrrc}
$\tcomp_{n,k}$ & $k=1$ & 2 & 3 & 4 & 5 & 6 & $\cdots$\\
\hline
$n=1$ & 1 \\
2 & 1 & 1 \\
3 & 4 & 3 & 1 &  \\
4 & 23 & 16 & 6 & 1 & \\
5 & 166 & 115 & 40 & 10 & 1 \\
6 & 1437 & 996 & 345 & 80 & 15 & 1 \\
$\vdots$ & $\vdots$ & $\vdots$ & $\vdots$ & $\vdots$ & $\vdots$ & $\vdots$ & $\ddots$
\end{tabular}
\caption{The count of labelled threshold graphs on vertex set $[n]$ with $k$ components.} \label{table-tcomp}
\end{center}
\end{table}

The Stirling number of the second kind $\stirling2{n}{k}$ (OEIS \seqnum{A008277}) counts partitions of a set of size $n$ into $k$ non-empty blocks (when we say ``partition'' from here on, we always mean ``partition into non-empty blocks''). There is an identity linking Eulerian numbers and Stirling numbers of the second kind (due originally to Frobenius \cite{Frobenius}, and discussed in \cite[Chapter 17.2]{MahadevPeled}): for $n \geq 1$, 
\begin{equation} \label{frobenius}
\sum_{k=0}^{n-1} \eulerian{n}{k}x^k = \sum_{\ell=1}^n \ell!\stirling2{n}{\ell}(x-1)^{n-\ell}.
\end{equation}
We will see in Section \ref{sec-lt-intro} that this identity is directly relevant to the enumeration of loop-threshold graphs. In Section \ref{sec-Frobenius-for-threshold}, we prove an analog of \eqref{frobenius} in the setting of labelled threshold graphs. Let $\op2_{n,\ell}$ denote the number of ordered partitions of $[n]$ (also known as {\it weak orders}) into $\ell$ blocks, in which the first (i.e., smallest) block has size at least $2$.
\begin{theorem} \label{thm-thresh-frob}
For $n \geq 2$,
\begin{equation} \label{eq-thresh-frob}
\sum_{k=1}^{n-1} (n-k)\eulerian{n-1}{k-1}x^{k-1}=\sum_{\ell=1}^{n-1} \op2_{n,\ell}(x-1)^{n-\ell-1}.
\end{equation}
\end{theorem}
Note the parallel between \eqref{frobenius} and \eqref{eq-thresh-frob}: the $\ell!\stirling2{n}{\ell}$ on the right-hand side of \eqref{frobenius} is the number of ordered partitions of $[n]$ into $\ell$ blocks, with no restriction on the size of the first block. Note also that setting $y=x+1$ in \eqref{eq-thresh-frob} and comparing coefficients of $y^{n-\ell-1}$ gives an exact formula for $\op2_{n,\ell}$, namely
\begin{equation} \label{op2-formula}
\op2_{n,\ell} = \sum_{k=1}^{n-1} (n-k)\eulerian{n-1}{k-1}\binom{k-1}{n-\ell-1}.
\end{equation}
The triangle $(\op2_{n,\ell}:n\geq 1, 1 \leq \ell \leq n)$ is shown in Table \ref{table-op2}, and is OEIS \seqnum{A348576}. The row sums of this triangle, $0, 1, 4, 23, 166, 1437, \ldots$, 
enumerate ordered partitions of a set of size $n$ ($n \geq 1$) into (any number of) non-empty blocks, in which the first block has size at least $2$. This is OEIS \seqnum{A053525}.

\begin{table}[ht!]
\begin{center}
\begin{tabular}{r|rrrrrrc}
$\op2_{n,\ell}$ & $\ell=1$ & 2 & 3 & 4 & 5 & 6 & $\cdots$ \\
\hline
$n=1$ & 0  \\
2 & 1 & 0  \\
3 & 1 & 3 & 0  \\
4 & 1 & 10 & 12 & 0  \\
5 & 1 & 25 & 80 & 60 & 0  \\
6 & 1 & 56 & 360 & 660 & 360 & 0  \\
$\vdots$ & $\vdots$ & $\vdots$ & $\vdots$ & $\vdots$ & $\vdots$ & $\vdots$ & $\ddots$
\end{tabular}
\caption{The count of ordered partitions of $[n]$ into $\ell$ non-empty blocks, in which the first block has size at least 2.} \label{table-op2}
\end{center}
\end{table}

As discussed earlier, labelled threshold graphs can be constructed iteratively, starting from a single vertex, by successively adding isolated or dominating vertices. Although there is not a unique such iterative construction,  
the number of dominating vertices added in any such construction of $G$ is independent of the choice of construction. Let $d(G)$ denote this number (note, we do not consider the initial vertex to be dominating). In Section \ref{sec-thresh-dom-proofs} we obtain an explicit expression for $d_{n,k}$, the number of labelled threshold graphs $G$ on vertex set $[n]$ with $d(G)=k$: $d_{n,0}=1$ for all $n \geq 1$, $d_{2,1}=1$, $d_{2,k}=0$ for all $k \geq 2$, and for $n \geq 3, k \geq 1$,
\begin{equation} \label{eq-dom-thresh}
d_{n,k}  = \begin{array}{l} \sum_{\ell \geq 1} \binom{n}{k+1}\op2_{k+1,\ell}\left[(\ell-1)!\stirling2{n-k-1}{\ell-1} + \ell!\stirling2{n-k-1}{\ell}\right] + \\
\\
\sum_{\ell \geq 1} \binom{n}{k}\ell!\stirling2{k}{\ell}\left[\op2_{n-k,\ell+1} + \op2_{n-k,\ell}\right]. 
\end{array}
\end{equation}
(Note that \eqref{op2-formula} gives a way of calculating the $\op2_{\cdot,\cdot}$ terms, in terms of more fundamental counting sequences.)
The triangle $(d_{n,k}:n \geq 1, 0 \leq k \leq n-1)$ is shown in Table \ref{table-d}, and is OEIS \seqnum{A350060}.

\begin{table}[ht!]
\begin{center}
\begin{tabular}{r|rrrrrrc}
$d_{n,k}$ & $k=0$ & 1 & 2 & 3 & 4 & 5 & $\cdots$ \\
\hline 
$n=1$ & 1  \\
2 & 1 & 1  \\
3 & 1 & 6 & 1  \\
4 & 1 & 22 & 22 & 1  \\
5 & 1 & 65 & 200 & 65 & 1  \\
6 & 1 & 171 & 1265 & 1265 & 171 & 1 \\
$\vdots$ & $\vdots$ & $\vdots$ & $\vdots$ & $\vdots$ & $\vdots$ & $\vdots$ & $\ddots$
\end{tabular}
\caption{The count of labelled threshold graphs on vertex set $[n]$ in which $k$ dominating vertices are added in an iterative construction.} 
\label{table-d}
\end{center}
\end{table}

\subsection{Labelled quasi-threshold graph results} \label{sec-lqtgr}

The family of {\it quasi-threshold graphs} is the smallest family of graphs that contains $K_1$, and is closed under adding dominating vertices and taking disjoint unions; it properly contains the family of threshold graphs. Like threshold graphs, quasi-threshold graphs have many alternate characterizations. One such is that they are precisely the comparability graphs of trees; it was in this context that they were first studied, by Wolk \cite{Wolk}. Another characterization is that quasi-threshold graphs are precisely those with the property that in each induced subgraph, the size of the largest independent set equals the number of maximal cliques. As observed by Golumbic \cite{Golumbic}, it is trivial to show that such graphs are perfect, and so quasi-threshold graphs are also called {\it trivially perfect graphs}.  

Let $\qt_n$ denote the number of labelled quasi-threshold graphs on $n$ vertices. As was possibly first explicitly observed by Guruswami \cite{Guruswami}, the sequence $(\qt_n)_{n \geq 1}$ begins as is shown in Table \ref{table-qt}, and is OEIS \seqnum{A058864}. 

\begin{table}[ht!]
\begin{center}
\begin{tabular}{r|rrrrrrc}
$n$ & 1 & 2 & 3 & 4 & 5 & 6 & $\cdots$  \\
\hline
$\qt_n$ & 1 & 2 & 8 & 49 & 402 & 4144 & $\cdots$ 
\end{tabular}
\caption{The count of labelled quasi-threshold graphs on vertex set $[n]$.} \label{table-qt}
\end{center}
\end{table}

In Section \ref{sec-qtn1-proof}, we give a combinatorial proof of the following formula that appears at OEIS \seqnum{A058864}: 
\begin{equation} \label{qtn1}
\qt_n = \sum_{k=0}^n (-1)^{n-k} \stirling2{n} {k}(k+1)^{k-1}.
\end{equation}
The proof uses the method of sign-changing 
involutions (see, for example, \cite{BenjaminQuinn}). Since this method will be used again in the sequel (to prove \eqref{ltn1-conj}), we digress here to give a brief description of it.

Suppose that we have a family $({\mathcal F}_n)_{n \geq 0}$ of sets, and that we want to verify the identity $|{\mathcal F}_n| = \sum_{k=0}^n (-1)^{n-k} g(n,k)$, where the $g(n,k)$ are non-negative integers for all $n, k \geq 0$. We proceed as follows.
\begin{description}
\item[Step 1] Find a set ${\mathcal G}_n$ whose size is $\sum_{k=0}^n g(n,k)$, that naturally partitions as ${\mathcal G}_n = \cup_{k=0}^n {\mathcal G}_{n,k}$, where $|{\mathcal G}_{n,k}|=g(n,k)$.
\item[Step 2] Find an involution $\iota:{\mathcal G}_n \rightarrow {\mathcal G}_n$ (a bijection satisfying $\iota \circ \iota = {\rm identity}$) with the following properties:
\begin{itemize}
\item all the fixed points of $\iota$ lie in ${\mathcal G}_n^{\rm even} := \cup_{k:~\! n-k~{\rm even}}~ {\mathcal G}_{n,k}$, and
\item all the orbits of size $2$ of $\iota$ intersect both ${\mathcal G}_n^{\rm even}$ and ${\mathcal G}_n \setminus {\mathcal G}_n^{\rm even}$.
\end{itemize}
Note that these conditions together imply that ${\mathcal G}_n^{\rm fixed}$, the set of fixed points of $\iota$, satisfies 
$$
|{\mathcal G}_n^{\rm fixed}| = \sum_{k=0}^n (-1)^{n-k} g(n,k).
$$
\item[Step 3] Find a bijection from ${\mathcal G}_n^{\rm fixed}$ to ${\mathcal F}_n$.
\end{description}

Guruswami \cite{Guruswami} used generating functions to obtain the following explicit formula for $\qtconn_n$, the number of connected labelled quasi-threshold graphs on $n$ vertices:
\begin{equation} \label{connqtn1}
\qtconn_n = \sum_{k=1}^n (-1)^{n+k} \stirling2{n}{k}k^{k-1}.
\end{equation}
As described in Section \ref{sec-qt-comp-proofs}, the proof of \eqref{qtn1} is easily modified to give a combinatorial proof of \eqref{connqtn1}. More generally, we give a simple combinatorial proof of the following: 
\begin{claim}
The number $\qtcomp_{n,\ell}$ of labelled quasi-threshold graphs on $n$ vertices with $\ell$ components is
\begin{equation} \label{compqtn}
\qtcomp_{n,\ell} = \sum_{k=1}^n (-1)^{n-k} \stirling2{n}{k}\ell\binom{k}{\ell}k^{k-\ell-1}.
\end{equation}
\end{claim}
The triangle $(\qtcomp_{n,\ell}:n \geq 1, 1 \leq \ell \leq n)$ is shown in Table \ref{table-qtcomp}, and is OEIS \seqnum{A350528}. The first column (enumerating connected labelled quasi threshold graphs, by number of vertices) is OEIS \seqnum{A058863}. 

\begin{table}[ht!]
\begin{center}
\begin{tabular}{r|rrrrrrc}
$\qtcomp_{n,\ell}$ & $\ell=1$ & 2 & 3 & 4 & 5 & 6 & $\cdots$ \\
\hline
$n=1$ & 1 \\
2 & 1 & 1 \\
3 & 4 & 3 & 1  \\
4 & 23 & 19 & 6 & 1 \\
5 & 181 & 155 & 55 & 10 & 1  \\
6 & 1812 & 1591 & 600 & 125 & 15& 1 \\
$\vdots$ & $\vdots$ & $\vdots$ & $\vdots$ & $\vdots$ & $\vdots$ & $\vdots$ & $\ddots$
\end{tabular}
\caption{The count of labelled quasi-threshold graphs on vertex set $[n]$ with $\ell$ components.}
\label{table-qtcomp}
\end{center}
\end{table}

\subsection{Labelled loop-threshold graph results} \label{sec-lt-intro}

All graphs that have been considered so far have been simple. For the last two families that we consider, loops are allowed, but not multiple edges. The two families are direct analogs of threshold and quasi-threshold graphs, in the world of looped graphs.

The family of {\it loop-threshold graphs} is the smallest family of graphs (finite, potentially with loops, but without multiple edges) that contains $K_1$ and $K_1^{\rm loop}$ (a loop on a single vertex), and is closed under adding isolated vertices and adding looped dominating vertices (looped vertices also adjacent to all other vertices in the graph).

Both loop-threshold and quasi-loop-threshold graphs (see Section \ref{sec-qlt-results}) were introduced by Cutler and Radcliffe \cite{CutRad2} as part of their study of extremal enumerative questions for graph homomorphisms. These families are natural to consider as target graphs in this context; for example, independent sets in graphs can be encoded as homomorphisms to the loop-threshold graph that is formed by adding a looped dominating vertex to an isolated vertex. Cutler and Radcliffe prove that for each $n$ and $m$ and loop-threshold graph $H$, among the $m$-edge $n$-vertex graphs that admit the most homomorphisms to $H$ there is at least one threshold graph; and if $H$ is quasi-loop-threshold, then at least one maximizer is quasi-threshold. 

Cutler and Radcliffe also observe that
loop-threshold and quasi-loop-threshold graphs have characterizations in terms of vertex neighborhoods that are exact analogs of well-known characterizations of threshold and quasi-threshold graphs. Specifically, a graph is threshold iff for each pair $x, y \in V(G)$ we have either $N(x)\setminus\{y\} \subseteq N(y)\setminus\{x\}$ or $N(y)\setminus\{x\} \subseteq N(x)\setminus\{y\}$, where $N(x)$ is the set of vertices adjacent to $x$ (see \cite{ChvatalHammer}), while it is loop-threshold iff for each pair $x, y \in V(G)$ we have either $N(x) \subseteq N(y)$ or $N(y) \subseteq N(x)$ (note $N(x)$ includes $x$ if there is a loop at $x$) (see \cite{CutRad2}). On the quasi- side, a graph is quasi-threshold iff there is a containment relation between $N(x)\setminus\{y\}$ and $N(y)\setminus\{x\}$ for each {\it adjacent} pair $x, y \in V(G)$ (see \cite{CutRad}), while it is loop-quasi-threshold iff there is a containment relation between $N(x)$ and $N(y)$ for each adjacent pair $x, y \in V(G)$ (see \cite{CutRad2}).    

Unlike for threshold graphs, little previous work has been done on enumeration of labelled loop-threshold graphs. Let $\lt_n$ denote the number of labelled loop-threshold graphs on $n$ vertices. The sequence $(\lt_n)_{n \geq 1}$ is shown in Table \ref{table-lt} (the first few terms can be found by inspection; later terms follow from our results.) 

\begin{table}[ht!]
\begin{center}
\begin{tabular}{r|rrrrrrc}
$n$ & 1 & 2 & 3 & 4 & 5 & 6 & $\cdots$ \\
\hline
$\lt_n$ & 2 & 6 & 26 & 150 & 1082 & 9366 & $\cdots$ 
\end{tabular}
\caption{The count of labelled loop-threshold graphs on vertex set $[n]$.} \label{table-lt}
\end{center}
\end{table}
  
In Section \ref{sec-lt-initial-proofs} we show that this sequence is OEIS \seqnum{A000629}, whose $n$th term counts the number of cyclically ordered partitions of $n+1$ objects. We do this by presenting a combinatorial proof of the formula 
\begin{equation} \label{eq-lt-basic}
\lt_n = 2\sum_{k=1}^n k!\stirling2{n}{k},
\end{equation}
which is given at OEIS \seqnum{A000629}. Note that the right-hand side of \eqref{eq-lt-basic} does indeed enumerate cyclically ordered partitions of $n+1$ objects. Indeed, $\sum_{k=1}^n k!\stirling2{n}{k}$ enumerates ordered partitions of $[n]$, and there is a 1-to-2 correspondence between ordered partitions of $[n]$, and cyclically ordered partitions of $[n+1]$ --- close the ordered partition into a cyclically ordered partition, and either let $n+1$ be in a block on its own between the last and first blocks of the ordered partition, or include $n+1$ in the first block. 

Setting $x=2$ in the Frobenius formula \eqref{frobenius},
we obtain from \eqref{eq-lt-basic} the following analog of \eqref{Spiro} for loop-threshold graphs:
\begin{equation} \label{ltn2-conj}
\lt_n = \sum_{k=0}^{n-1} \eulerian{n}{k}2^{k+1}.
\end{equation}
The right-hand side above appears at OEIS \seqnum{A000629}. (Actually,  \eqref{ltn2-conj} appears at OEIS \seqnum{A000629} with $2^{k+1}$ replaced by $2^k$; this discrepancy is explained by the fact that at OEIS \seqnum{A000629}, the Eulerian number $\eulerian{n}{k}$ is counting permutations of $[n]$ with $k-1$ ascents rather than with $k$.)  
In Section \ref{sec-lt-Frob} we give a short direct combinatorial proof of \eqref{ltn2-conj} (as opposed to one passing through the Frobenius formula). It is similar in spirit to Spiro's proof of \eqref{Spiro}.

Our proof of \eqref{ltn2-conj} is very similar to a proof of a related identity involving Eulerian numbers and Stirling numbers given by Bona \cite[Section 1.1.3]{Bona}, and generalizes easily to a combinatorial proof of the full Frobenius formula, that is different from the semi-combinatorial proof given in \cite{BeissingerPeled}. For completeness, we also include this extension in Section \ref{sec-lt-Frob}. 

Another formula suggested by OEIS \seqnum{A000629} is the alternating sum 
\begin{equation} \label{ltn1-conj}
\lt_n = \sum_{k=0}^n (-1)^{n-k}\stirling2{n}{k}k!2^k.
\end{equation}
The right-hand sides of \eqref{ltn2-conj} and \eqref{ltn1-conj} can be seen to be equal using the Frobenius identity \eqref{frobenius} at $x=1/2$ (and using the symmetry $\eulerian{n}{\ell}=\eulerian{n}{n-1-\ell}$ of the Eulerian numbers). In Section \ref{sec-lt-involution} we use a sign-changing involution to give a direct combinatorial proof of \eqref{ltn1-conj}.

Next, we consider the enumeration of loop-threshold graphs by number of components. Let $\ltcomp_{n,k}$ denote the number of labelled loop-threshold graphs on vertex set $[n]$ with $k$ components. In Section \ref{sec-lt-comp} we give the easy derivation that, for $n, k \geq 1$:
\begin{enumerate}[(i)]
\item $\ltcomp_{n,1} = \sum_{k=0}^{n-1} \eulerian{n}{k}2^k$ (so $\ltcomp_{n,1}=\frac{\lt_n}{2}$ for $n \geq 2$),
\item for $n \geq 3$ and $2 \leq k \leq n-1$, $\ltcomp_{n,k} = \binom{n}{k-1} \ltcomp_{n-k+1,1}$, and
\item for $n \geq 2$, $\ltcomp_{n,n}=n+1$. 
\end{enumerate}

The triangle $(\ltcomp_{n,k}:n \geq 1, 1 \leq k \leq n)$ is shown in Table \ref{table-ltcomp}, and is  OEIS \seqnum{A350531}. This is very close to OEIS \seqnum{A154921} --- in the triangle shown in Table \ref{table-ltcomp}, replace the last non-zero entry ($m+1$) of each row with two entries $(m,1)$ to get OEIS \seqnum{A154921}. The first column of the triangle, the sequence of connected labelled loop-threshold graphs by number of vertices, is, apart from an anomaly at $n=1$, OEIS \seqnum{A000670}. 

\begin{table}[ht!]
\begin{center}
\begin{tabular}{r|rrrrrrc}
$\ltcomp_{n,k}$ & $k=1$ & 2 & 3 & 4 & 5 & 6 & $\cdots$ \\
\hline
$n=1$ & 2 \\
2 & 3 & 3 \\
3 & 13 & 9 & 4 \\
4 & 75 & 52 & 18 & 5 \\
5 & 541 & 375 & 130 & 30 & 6 \\
6 & 4683 & 3246 & 1125 & 260 & 45 & 7 \\
$\vdots$ & $\vdots$ & $\vdots$ & $\vdots$ & $\vdots$ & $\vdots$ & $\vdots$ & $\ddots$
\end{tabular}
\caption{The count of labelled loop-threshold graphs on vertex set $[n]$ with $k$ components.}
\label{table-ltcomp}
\end{center}
\end{table}

Just as with threshold graphs, we can associate to any loop-threshold graph $G$ an invariant $\ld(G)$, the number of looped dominating vertices added in any iterative construction of $G$. This takes values between $0$ and $|V(G)|$, the number of vertices of $G$, and unlike the analogous invariant for threshold graphs, $\ld(G)$ can easily be measured without considering the iterative construction --- it is the number of looped vertices of $G$. Notice also that in contrast to the case of threshold graphs, we do not have to make an arbitrary decision regarding the first vertex in the construction.

Let $\ld_{n,k}$ denote the number of labelled loop-threshold graphs $G$ on vertex set $[n]$ with $\ld(G)=k$. Notice that  $\ld_{n,0}=1$ for all $n \geq 1$.  By a very similar argument to the one used to derive \eqref{eq-dom-thresh} (presented in Section \ref{sec-thresh-dom-proofs}), we can 
derive that for $n \geq 2, k \geq 1$, we have
$$
\ld_{n,k}= \binom{n}{k}\sum_{\ell \geq 1} \ell!\stirling2{k}{\ell}h(\ell,n,k).
$$
where
$$
h(\ell,n,k)=(\ell-1)!\stirling2{n-k}{\ell-1} + 2\ell!\stirling2{n-k}{\ell} + (\ell+1)!\stirling2{n-k}{\ell+1}.
$$
(We omit the proof.) The triangle $(\ld_{n,k}:n \geq 1, 0 \leq k \leq n)$ is shown in Table \ref{table-ld}, and is OEIS \seqnum{A350745}.

\begin{table}[ht!]
\begin{center}
\begin{tabular}{r|rrrrrrrc}
$\ld_{n,k}$ & $k=0$ & 1 & 2 & 3 & 4 & 5 & 6 & $\cdots$ \\
\hline
$n=1$ & 1 & 1 \\
2 & 1 & 4 & 1  \\
3 & 1 & 12 & 12 & 1 \\
4 & 1 & 32 & 84 & 32 & 1  \\
5 & 1 & 80 & 460 & 460 & 80 & 1  \\
6 & 1 & 192 & 2190 & 4600 & 2190 & 192 & 1 \\
$\vdots$ & $\vdots$ & $\vdots$ & $\vdots$ & $\vdots$ & $\vdots$ & $\vdots$ & $\vdots$ & $\ddots$
\end{tabular}
\caption{The count of labelled loop-threshold graphs on vertex set $[n]$ with $k$ looped vertices.}
\label{table-ld}
\end{center}
\end{table}

\subsection{Labelled quasi-loop-threshold graph results} \label{sec-qlt-results}

The family of {\it quasi-loop-threshold graphs} is the smallest family of graphs (finite, potentially with loops, but without multiple edges) that contains $K_1$ and $K_1^{\rm loop}$ (a loop on a single vertex), and is closed under taking disjoint unions and adding looped dominating vertices; it properly contains the family of loop-threshold graphs. (See the start of Section \ref{sec-lt-intro} for a little more on this family). As with loop-threshold graphs, little previous work has been done on enumeration of labelled quasi-loop-threshold graphs.

Let $\qlt_n$ denote the number of labelled quasi-loop-threshold graphs on $n$ vertices. The sequence $(\qlt_n)_{n \geq 1}$ is shown in Table \ref{table-qlt} (the first few terms can be found by inspection; later terms follow from our results.)

\begin{table}[ht!]
\begin{center}
\begin{tabular}{r|rrrrrrc}
$n$ & 1 & 2 & 3 & 4 & 5 & 6 & $\cdots$ \\
\hline
$\qlt_n$ & 2 & 7 & 42 & 376 & 4513 & 68090 & $\cdots$ 
\end{tabular}
\caption{The count of labelled quasi-loop-threshold graphs on vertex set $[n]$.} \label{table-qlt}
\end{center}
\end{table}

In Section \ref{sec-quasi-loop-thresh-proofs} we show that this sequence is an offset of OEIS \seqnum{A038052}, whose $n$th term counts the number of labeled trees whose vertices are labelled with the blocks of a partition of $[n]$. Specifically, we show that $\qlt_n$ is the $(n+1)$st term of \seqnum{A038052}. We do this by presenting a combinatorial proof of the formula 
\begin{equation} \label{qltn1}
\qlt_n = \sum_{k=1}^{n+1} \stirling2{n+1}{k} k^{k-2}.
\end{equation}
Note that by Cayley's formula, the right-hand side of \eqref{qltn1} (which is given at OEIS \seqnum{A038052}) is easily seen to enumerate labeled trees whose vertices are labelled with the blocks of a partition of $[n+1]$.

We conclude our note by considering the enumeration of labelled quasi-loop-threshold graphs by number of components. Here things are not as clean as in the case of quasi-threshold graphs.
In Section \ref{sec-quasi-loop-thresh-proofs} we consider the number $\qltconn_n$ of connected labelled quasi-loop-threshold graphs on $n$ vertices (this turns out to be a useful first step in obtaining \eqref{qltn1}), and show that $\qltconn_1=2$ and for $n \geq 2$  
\begin{equation} \label{qlt-conn}
\qltconn_n = \sum_{k=1}^n \stirling2{n}{k}
k^{k-1}.
\end{equation}
Except for an anomaly at $n=1$, the sequence $(\qltconn_n)_{n \geq 1}$ is OEIS \seqnum{A048802} (that sequence takes value $2$ at $n=1$, as opposed to $1$; so OEIS \seqnum{A048802} counts labelled quasi-loop-threshold on $n$ vertices that have at least one looped dominating vertex.)  

Let $\qltcomp_{n,\ell}$ denote the number of labelled quasi-loop-threshold on $n$ vertices with $\ell$ components. It follows via standard considerations that
\begin{equation} \label{qlt-comp}
\qltcomp_{n,\ell} = \frac{1}{\ell!} \sum_{x_1+\cdots +x_\ell=n,~x_i \geq 1} \binom{n}{x_1,\ldots, x_\ell} qlt^{\rm conn}_{x_1}\cdots qlt^{\rm conn}_{x_\ell}.
\end{equation}
We also have, via the exponential formula (see e.g. \cite{Wilf}), that 
$$
\qltcomp_{n,\ell}=\frac{n!}{\ell!}[x^n]\left(\sum_{k \geq 1}\frac{\qltconn_k x^k}{k!}\right)^\ell,
$$
where $[x^n]f(x)$ denotes the coefficient of $x^n$ in the power series $f(x)$. The triangle $(\qltcomp_{n,k}:n \geq 1, 1 \leq k \leq n)$ is shown in Table \ref{table-qltcomp}, and is OEIS \seqnum{A350746}.

\begin{table}[ht!]
\begin{center}
\begin{tabular}{r|rrrrrrc}
$\qltcomp_{n,\ell}$ & $\ell=1$ & 2 & 3 & 4 & 5 & 6 & $\cdots$\\
\hline
$n=1$ & 2 \\
2 & 3 & 4 \\
3 & 16 & 18 \\
4 & 133 & 155 & 72 & 16 \\
5 & 1521 & 1810 & 910 & 240 & 32 \\
6 & 22184 & 26797 & 14145 & 4180 & 720 & 64 \\
$\vdots$ & $\vdots$ & $\vdots$ & $\vdots$ & $\vdots$ & $\vdots$ & $\vdots$ & $\ddots$
\end{tabular}
\caption{The count of labelled quasi-loop-threshold on vertex set $[n]$ with $k$ components.}
\label{table-qltcomp}
\end{center}
\end{table}

\section{Proofs for threshold graph results} \label{sec-thresh-proofs}

\subsection{Spiro's proof of \eqref{Spiro}} \label{subsec-spiros-proof}

Our goal here is to establish that for $n \geq 2$ the number $t_n$ of labelled threshold graphs on $[n]$ satisfies
\begin{equation} \label{Spiro'}
t_n = \sum_{k=1}^{n-1}(n-k)\eulerian{n-1}{k-1}2^k.
\end{equation}

Threshold graphs can be constructed from $K_1$ by iteratively adding isolated or dominating vertices, and so a labelled threshold graph $G$ on vertex set $[n]$, $n \geq 2$, can be uniquely encoded by a pair $(s,P)$, where 
\begin{itemize}
\item $s \in \{+,-\}$ and 
\item $P$ is an ordered partition of $[n]$ with the first block having size at least $2$.
\end{itemize}
Indeed, if the construction of $G$ from $K_1$ begins by assigning a label to $K_1$ and then adding some (labelled) dominating vertices, we can take $s=+$ and take as the first block of $P$ the label assigned to $K_1$ together with the labels of all the added dominating vertices that are added before the first isolated vertex is added; then take the second block of $P$ to be the labels of all the isolated vertices that are added before the next dominating vertex is added, and so on. If the construction of $G$ from $K_1$ begins by assigning a label to $K_1$ and then adding some isolated vertices, we can take $s=-$ and then proceed similarly. It is easily seen that each threshold graph has a unique code, and that for each pair $(s, P)$ there is a unique threshold graph that has $(s, P)$ as its code. (See Figure \ref{fig-thresh-codes}.) 

\begin{figure}[ht!] 
\begin{center}
\begin{tikzpicture}
\node[label=below:2] at (0,0) [circle,fill,inner sep=1.5pt]{};
\node[label=below:3] at (1,0) [circle,fill,inner sep=1.5pt]{};
\node[label=below:1] at (2,.5) [circle,fill,inner sep=1.5pt]{};
\node[label=below:4] at (3,1.5) [circle,fill,inner sep=1.5pt]{};
\node[label=below:5] at (4,3) [circle,fill,inner sep=1.5pt]{};
\node[label=below:1] at (7,0) [circle,fill,inner sep=1.5pt]{};
\node[label=below:3] at (8,0) [circle,fill,inner sep=1.5pt]{};
\node[label=below:5] at (9,.5) [circle,fill,inner sep=1.5pt]{};
\node[label=below:2] at (10,1.5) [circle,fill,inner sep=1.5pt]{};
\node[label=below:4] at (11,3) [circle,fill,inner sep=1.5pt]{};
\draw (1,0) -- (0,0);
\draw (4,3) -- (0,0);
\draw (4,3) -- (1,0);
\draw (4,3) -- (2,.5);
\draw (4,3) -- (3,1.5);
\draw (10,1.5) -- (7,0);
\draw (10,1.5) -- (8,0);
\draw (10,1.5) -- (9,.5);
\draw (11,3) -- (7,0);
\draw (11,3) -- (8,0);
\draw (11,3) -- (9,.5);
\draw (11,3) -- (10,1.5);
\end{tikzpicture}
\caption{Two labelled threshold graphs. The one on the left has code $(+,23/14/5)$, while the one on the right has code $(-,135/24)$.}
\label{fig-thresh-codes}
\end{center}
\end{figure}
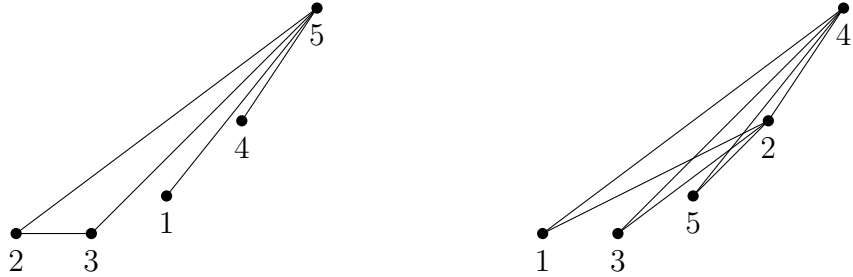

So, let ${\mathcal O}_n$ be the set of all pairs $(s,P)$ with $s\in\{+,-\}$ and with $P$ a partition of $[n]$ with first block having size at least $2$, and let ${\mathcal P}_n$ be the set of triples $(s,\pi,c)$ where 
\begin{itemize}
\item $s \in \{+,-\}$, 
\item $\pi$ is a permutation of $[n]$ that starts with an ascent, and 
\item $c$ is an assignment of a color, from a palette of two colors, say red and blue, to each of the ascents of $\pi$ other than the first ascent.
\end{itemize}

\begin{claim} \label{clm-t-eul}
$|{\mathcal O}_n|=|{\mathcal P}_n|$.
\end{claim}

\medskip

Before proving Claim \ref{clm-t-eul} we complete the combinatorial proof of \eqref{Spiro} by observing that the right-hand side of \eqref{Spiro'} enumerates ${\mathcal P}_n$. Indeed, there are $(n-k)\eulerian{n-1}{k-1}$ permutations of $[n]$ that have $k$ ascents and that start with an ascent (this is established in \cite[Lemma 6]{Spiro}, and is also observed by Gessel at OEIS \seqnum{A008292}). For each such permutation, there are $2^{k-1}$ ways to $2$-color the ascents (other than the first ascent); the extra factor of $2$ in the summand on the right-hand side of \eqref{Spiro'} allows for the selection of $s$.

\medskip

A natural operation on ordered partitions comes up in the proof of Claim \ref{clm-t-eul}, and will recur throughout other proofs, so we record it as a definition.
\begin{definition} \label{def-flat}
For an ordered partition $P=P_1/\cdots/P_k$ with, for each $i$,  $P_i=\{p_{i1},\ldots,p_{i{\ell_i}}\}$ ($p_{i1} <\cdots <p_{i{\ell_i}}$), the {\em flattening} of $P$ is the permutation 
$$
\pi_f(P) = p_{11}\cdots p_{1{\ell_1}} \cdots p_{k1}\cdots p_{k{\ell_k}}
$$
of $\cup_{i=1}^k P_i$.
\end{definition}
For example, for $P=57/321/64$ (an ordered partition of $[7]$), we have $\pi_f(P) = 5712346$.

\begin{proof}[Proof (of Claim \ref{clm-t-eul})]
We describe a bijection from ${\mathcal O}_n$ to ${\mathcal P}_n$. Given a pair $(s,P) \in {\mathcal O}_n$,
set $\pi=\pi_f(P)$ (see Definition \ref{def-flat}). Because the first block of $P$ has size at least $2$, $\pi$ starts with an ascent. Color an ascent (other than this initial one) red if it is not the largest element of its block, and blue otherwise (note that ascents colored blue are necessarily the largest elements in their blocks). The triple $(s,\pi,c)$ is in ${\mathcal P}_n$. 
\begin{example} \label{ex-ascent}
The pair $(+,235/17/46/8) \in {\mathcal O}_n$ maps to the triple
$$
(+,2\textcolor{red}{3}5\textcolor{red}{1}7\textcolor{red}{4}\textcolor{blue}{6}8,c) \in {\mathcal P}_n,
$$ 
where $c^{-1}({\rm red})=\{3,1,4\}$ and $c^{-1}({\rm blue})=\{6\}$.
\end{example}

To see that the map described above from ${\mathcal O}_n$ to ${\mathcal P}_n$ is a bijection, we now describe how to invert it. Given a triple $(s,\pi,c) \in {\mathcal P}_n$, form an ordered partition $P$ of $[n]$ as follows: write $\pi$ in one-line notation, make a break in $\pi$ after each ascent colored blue, and after each descent. Let $P_1$ (the first block of $P$) be the set of symbols that occur before the first break, $P_2$ (the second block of $P$) be the set of symbols that occur between the first and second breaks, et cetera. Then the pair $(s,P)$ is in ${\mathcal O}_n$,
and is easily seen to map to $(s,\pi,c)$ under the map described in paragraph preceding Example \ref{ex-ascent}.
\end{proof}

\subsection{Proof of \eqref{tn1}} \label{subsec-tn1-proof}

The goal here is to establish that the number $t_n$ of labelled threshold graphs on vertex set $[n]$ satisfies
$$
t_n = 1-n +\sum_{k=0}^{n-1} \binom{n}{k}t_k
$$
for $n \geq 1$, with $t_0=1$.

The identity is evident for $n = 1, 2$, so we consider $n \geq 3$, for which the identity is equivalent to
\begin{equation} \label{tn1'}
t_n = 2+\sum_{k=2}^{n-1} \binom{n}{k}t_k.
\end{equation}
We use that $t_n=|{\mathcal O}_n|$ for $n \geq 2$, where ${\mathcal O}_n$ is the set of pairs $(s,P)$ with $s \in \{+,-\}$ and with $P$ an ordered partition of $[n]$ in which the first block has at least two elements (see Section \ref{subsec-spiros-proof}). This equality fails for $n \leq 1$, which explains the need here (and later) to deal with small $n$ instances of the recurrence by hand rather than via a general bijection. 

There is one ordered partition of $[n]$ with only one block, and this partition appears in two elements of ${\mathcal O}_n$ (one with $s=+$ and one with $s=-$). This accounts for the $2$ on the right-hand side of \eqref{tn1'}. 

For all other ordered partitions, the union of all the blocks but the final one has size ranging from $2$ to $n-1$; let this size be $k$ (this is the index of summation in the sum on the right-hand side of \eqref{tn1'}). The number of pairs $(s,P)$ corresponding to each $k$ is $\binom{n}{k}$ (select the set $L$ of labels that do not appear in the final block) times $t_k$ (select a pair $(s,P')$, where $s \in \{+,-\}$ and $P'$ is an ordered partition of $L$ in which the first block has at least two elements); note that each $(s,P')$ gives rise to a unique element of ${\mathcal O}_n$ by appending $[n]\setminus L$ as a block at the end of $ P'$. This establishes \eqref{tn1'}.  

\subsection{Proof of \eqref{tn2}} \label{subsec-tn2-proof}

The goal here is to establish that the number $t_n$ of labelled threshold graphs on vertex set $[n]$ satisfies
$$
t_{n+1} = n(t_n-t_{n-1}) +\sum_{k=0}^n \binom{n}{k}t_kt_{n-k}
$$
for $n \geq 1$, with $t_0=t_1=1$. The $k=0$ term in the sum
is $t_n$ and the $k=1$ term is $nt_{n-1}$, so the recurrence can
be rewritten as
\begin{equation} \label{tn2'}
t_{n+1}=(n+1)t_n + \sum_{k=2}^n \binom{n}{k}t_kt_{n-k}.
\end{equation}
This identity is easy to verify for $n=1,2$, so from here on we assume $n \geq 3$. As in Section \ref{subsec-tn1-proof}, we use that $t_n=|{\mathcal O}_n|$ for $n \geq 2$, where ${\mathcal O}_n$ is the set of pairs $(s,P)$ with $s \in \{+,-\}$ and with $P$ an ordered partition of $[n]$ in which the first block has at least two elements. To verify \eqref{tn2'} we will show that 
\begin{enumerate}[(a)]
\item $(n+1)t_n$ enumerates those pairs in ${\mathcal O}_{n+1}$ in which the partition $P$ ends with a block of size $1$, and  
\item the summation on the right-hand side of \eqref{tn2'} enumerates those pairs in which $P$ ends with a block of size at least $2$. 
\end{enumerate}
For (a), here is a bijection from pairs $(s,P)$ in ${\mathcal O}_{n+1}$ in which the partition $P$ ends with a block of size $1$, to triples $(a,s',P')$ where $a \in [n+1]$, $s' \in \{+,-\}$ and $P'$ is an ordered partition of $[n]$ in which the first block has at least two elements (note that the number of such triples is $(n+1)t_n$): given $(s,P)$, let $a$ be the number in the final block of $P$, let $s'=s$, and let $P''$ be obtained from $P$ by removing the final block. Here $P''$ is an ordered partition not of $[n]$ but of a set $S$ of $n$ integers; turn it into $P'$, an ordered partition of $[n]$ (whose first block still has size at least $2$) by replacing the $i$th largest element of $S$ by $i$, for $i=1, \ldots, n$. (We refer to this final operation, which we will need again later, as {\it compressing} $S$.) This map is easily seen to be invertible. 

\begin{example} \label{ex-tn1}
The pair $(+,137/6/25/4)$ maps to the triple $(4,+,136/5/24)$.
\end{example}
 
For (b): to an ordered partition $P$ of a set of integers, associate the permutation $\pi=\pi_f(P)$ (see Definition \ref{def-flat}). For $k=2, \ldots, n-2, \widehat{n-1}, n$ let ${\mathcal O}^k_{n+1}$ be the set of  pairs in ${\mathcal O}_{n+1}$ in which the last block of $P$ has size at least $2$, and in which $n+1$ appears in position $k+1$ in $\pi(P)$ (here $\widehat{\cdot}$ is indicating a missing element in a sequence). For $k=n-1$, let ${\mathcal O}^k_{n+1}~(={\mathcal O}^{n-1}_{n+1})$ be the set of pairs in ${\mathcal O}_{n+1}$ in which the last block of $P$ has size at least $2$, and in which $n+1$ appears in position $2$ in $\pi(P)$. We claim that
\begin{enumerate}[(i)]
\item $\cup_{k=2}^n {\mathcal O}^k_{n+1}$ is exactly the set of pairs in ${\mathcal O}_{n+1}$ in which the last block of $P$ has size at least $2$, and
\item for each $k=2, \ldots, n$, $|{\mathcal O}^k_{n+1}|=\binom{n}{k}t_kt_{n-k}$.
\end{enumerate}
Noting that the union in (i) is clearly a disjoint union, (i) and (ii) together yield (b), and this together with (a) yields \eqref{tn2'}. 

To see (i), note that there are no pairs in ${\mathcal O}_{n+1}$ in which the last block of $P$ has size at least $2$, and in which $n+1$ appears in position $1$ in $\pi(P)$ --- the first block of $P$ has at least two elements (by definition of ${\mathcal O}_{n+1}$), and $\pi(P)$ begins with the smallest element of the first block, which will not be $n+1$. Nor are there pairs in ${\mathcal O}_{n+1}$ in which the last block of $P$ has size at least $2$, and in which $n+1$ appears in position $n$ in $\pi(P)$ --- this would put $n+1$ in the last block of $P$, and so by construction of $\pi(P)$ would put it in position $n+1$. So the possible positions for $n+1$ are $2$ (counted by ${\mathcal O}^{n-1}_{n+1}$) and $3, \ldots, \widehat{n}, n+1$ (counted by ${\mathcal O}^k_{n+1}$ for $k=2, \ldots, \widehat{n-1}, n$).

To see (ii), we first consider ${\mathcal O}^k_{n+1}$ for $k \neq n-1, n$. To show $|{\mathcal O}^k_{n+1}|=\binom{n}{k}t_kt_{n-k}$ we associate to each $(s,P) \in {\mathcal O}^k_{n+1}$ a quintuple $(A,s',P',s'',P'')$ where $A$ is a subset of $[n]$ of size $k$, $(s',P') \in {\mathcal O}_k$ and $(s'',P'') \in {\mathcal O}_{n-k}$, in such a way that all such quintuples are seen once as $(s,P)$ varies over ${\mathcal O}^k_{n+1}$ (note that the number of such quintuples is $\binom{n}{k}t_kt_{n-k}$). We do this as follows. Given $(s,P) \in {\mathcal O}^k_{n+1}$,
\begin{itemize}
\item take $A$ to be the elements that precede $n+1$ in $\pi(P)$; note $|A|=k$.
\item Take $s'=s$.
\item To obtain $P'$, take all the blocks of $P$ up to and including the one that includes $n+1$. Let this block (i.e., the one containing $n+1$) be $f(P)$. Remove $n+1$ from $f(P)$ (or, if $f(P)$ is a singleton, remove the entire block). The result is an ordered partition of the set $A$, whose first block has size at least $2$; turn it into $P'$, an ordered partition of $[k]$ whose first block has size at least $2$, by compressing $A$ (see just before Example \ref{ex-tn1}).  
\item To obtain $P''$, take all the blocks of $P$ that occur after $f(P)$, in reverse order (i.e., starting with the final block of $P$, and ending with the block that comes immediately after $f(P)$). The result is an ordered partition of a set $T$ of integers of size $n-k$, whose first block has size at least $2$ (since the last block of $P$ has size at least $2$); turn this into $P''$, an ordered partition of $[n-k]$ whose first block has size at least $2$, by compressing $T$. 
\item Finally, we describe $s''$, whose function is to record whether $f(P)$ (recall, this is the block of $P$ that contains $n+1$) is a singleton or not; specifically, take $s''=+$ if $f(P)=\{n+1\}$ and $s''=-$ otherwise. 
\end{itemize}
\begin{example} \label{ex-tn2}
The pair $(-,37/129/6/4/58) \in {\mathcal O}^4_9$ maps to the quintuple 
$$
(\{1,2,3,7\},-,34/12,-,24/1/3),
$$ 
while $(-,37/12/9/6/4/58)$ maps to the quintuple
$$
(\{1,2,3,7\},-,34/12,+,24/1/3).
$$ 
\end{example}

Given a quintuple $(A,s',P',s'',P'')$ with $A$ a subset of $[n]$ of size $k$, $(s',P') \in {\mathcal O}_k$ and $(s'',P'') \in {\mathcal O}_{n-k}$, we can easily invert the above association to find the unique $(s,P) \in {\mathcal O}^k_n$ that gets associated with $(A,s',P',s'',P'')$; so $|{\mathcal O}^k_n|=\binom{n}{k}t_kt_{n-k}$. 

Next we deal with $k=n$, which is a little simpler than $k=2, \ldots, n-2$. We need to establish $|{\mathcal O}^n_{n+1}|=t_n~(=|{\mathcal O}_n|)$. Given $(s,P) \in {\mathcal O}^n_{n+1}$, in which (by definition of ${\mathcal O}^n_{n+1}$) $n+1$ is in the final block of $P$, simply remove $n+1$ to obtain $P'$, an ordered partition of $[n]$ in which the first block has size at least $2$. The correspondence $(s,P)\rightarrow(s,P')$ is clearly a bijection from ${\mathcal O}^n_{n+1}$ to ${\mathcal O}_n$. 

Finally we turn to the case $k=n-1$. We aim to show $|{\mathcal O}^{n-1}_{n+1}|=nt_{n-1}$. For $(s, P) \in {\mathcal O}^{n-1}_{n+1}$, note that since $n+1$ is the second element of $\pi(P)$ it must be that $P$ begins with a block of size $2$ that includes $n+1$ and one other number, say $a$. Associate with $(s,P)$ the triple $(a, s, P')$ where $P'$ is obtained from $P$ by removing the first block, then ordering the blocks in reverse order (so $P'$ starts with a block of size at least $2$) and then compressing $[n+1]\setminus\{a\}$. 
\begin{example} \label{ex-tn3}
The pair $(+,49/238/7/156) \in {\mathcal O}_9^7$ maps to $(4,+,145/6/237)$.
\end{example}
This association is clearly a bijection from ${\mathcal O}^{n-1}_{n+1}$ to $[n] \times {\mathcal O}_{n-1}$, completing the proof of \eqref{tn2'}.

\subsection{Labelled threshold graphs by number of components} \label{sec-thresh-comp-proofs}

For $n, k \geq 1$, recall that we let $\tcomp_{n,k}$ be the number of labelled threshold graphs on vertex set $[n]$ with $k$ components. The goal of this section is to explain the (quite straightforward) derivations from \eqref{Spiro} of the following formulas: 
\begin{enumerate}[(i)]
\item $\tcomp_{1,1} = 1$,
\item for $n \geq 2$,  $\tcomp_{n,1} = \sum_{k=1}^{n-1} (n-k)\eulerian{n-1}{k-1}2^{k-1}$,
\item for $n \geq 3$ and $2 \leq k \leq n-1$, $\tcomp_{n,k} = \binom{n}{k-1} \tcomp_{n-k+1,1}$ and
\item for $n \geq 2$, $\tcomp_{n,n}=1$.
\end{enumerate}

Evidently the only connected labelled threshold graph on one vertex is connected, and for $n \geq 2$ the only connected labelled threshold graph with $n$ components is the edgeless graph, so (i) and (iv) hold. 

For $n \geq 2$, we claim that half of all labelled threshold graphs on $[n]$ are connected. Indeed, the involution $\iota$ on ${\mathcal O}_n$ given by $\iota(+,P)=(-,P)$ and $\iota(-,P)=(+,P)$ has orbits that consist of one connected labelled threshold graph on $[n]$ ($(+,P)$ if $P$ has an odd number of blocks --- and so the construction of a threshold graph from $(+,P)$ ends by adding dominating vertices --- and $(-,P)$ if $P$ has an even number of blocks), and one graph that is not connected. Together with \eqref{Spiro}, this observation yields (ii). (Note that the involution $\iota$ interpreted in terms of threshold graphs corresponds to graph complementation.)

For $n \geq 3$ and $2 \leq k \leq n-1$, the set of labelled threshold graphs on vertex set $[n]$ with $k$ components is obtained by choosing $k-1$ elements from $[n]$, choosing a connected labelled threshold graph on the remaining $n-k+1$ elements, and then adding the chosen $k-1$ elements as isolated vertices (i.e., as a block added to the end of $P$); this gives (iii). 

\subsection{Proof of Theorem \ref{thm-thresh-frob}} \label{sec-Frobenius-for-threshold}

The goal here is to establish the following analog of the Frobenius formula \eqref{frobenius}: for $n \geq 2$,
\begin{equation} \label{eq-thresh-frob'}
\sum_{k=1}^{n-1} (n-k)\eulerian{n-1}{k-1}x^{k-1}=\sum_{\ell=1}^{n-1} \op2_{n,\ell}(x-1)^{n-\ell-1},
\end{equation}
where $\op2_{n,\ell}$ is the number of ordered partitions of $[n]$ into $\ell$ blocks, with the first block having size at least $2$.

For positive integers $x$, the right-hand side of \eqref{eq-thresh-frob} counts ordered partitions of $[n]$ into blocks, in which the first block has size at least $2$, and in which all elements are colored with one of $x-1$ colors, except for the smallest element in the first block and the largest element in each of the blocks (none of which is given a color).

To each such colored partition, associate a colored permutation as follows. The permutation is obtained from the ordered partition by flattening (see Definition \ref{def-flat}). Each element that has been colored retains its color; additionally, the entries in the permutation that correspond to the largest entries in a block, and that are also ascents, are given an $x$th color. 

The process just described yields a bijective correspondence from colored partitions to permutations of $[n]$ that start with an ascent and in which all ascents except the first are given one of $x$ colors. Indeed, given such a colored permutation, we can recover the colored partition it came from by locating the ascents colored $x$ and the (uncolored) descents --- these are the largest elements of each of the blocks of the partition --- and ignoring the color $x$. 

Since (see Section \ref{subsec-spiros-proof}) the number of permutations of $[n]$ that start with an ascent and have $k$ ascents is $(n-k)\eulerian{n-1}{k-1}$, we have bijectively established \eqref{eq-thresh-frob'} for all positive integers $x$; the identity follows for all $x$ since both sides of \eqref{eq-thresh-frob'} are polynomials in $x$.

\subsection{Proof of \eqref{eq-dom-thresh}} \label{sec-thresh-dom-proofs}

Recall that for a labelled threshold graph $G$, we let $d(G)$ denote the number of dominating vertices added in any iterative  construction of $G$ that starts from a single vertex and successively adds dominating or isolated vertices, and we let $d_{n,k}$ be the number of labelled threshold graphs $G$ on vertex set $[n]$ for which $d(G)=k$. The goal of this section is to establish $d_{n,0}=1$ for all $n \geq 1$, $d_{2,1}=1$, and for $n \geq 3, k \geq 1$,
\begin{equation} \label{eq-dom-thresh'}
d_{n,k}  = \begin{array}{l} \sum_{\ell \geq 1} \binom{n}{k+1}\op2_{k+1,\ell}\left[(\ell-1)!\stirling2{n-k-1}{\ell-1} + \ell!\stirling2{n-k-1}{\ell}\right] + \\
\\
\sum_{\ell \geq 1} \binom{n}{k}\ell!\stirling2{k}{\ell}\left[\op2_{n-k,\ell+1} + \op2_{n-k,\ell}\right]. 
\end{array}
\end{equation}

The boundary conditions ($d_{n,0}=1$ for all $n \geq 1$ and $d_{2,1}=1$) are clear. For the recurrence: 
identifying a labelled threshold graph with an element of ${\mathcal O}_n$ --- the set of pairs $(s,P)$ with $s \in \{+,-\}$ and with $P$ an ordered partition of $[n]$ with the first block having size at least $2$ (see Section \ref{subsec-spiros-proof}) --- we can read off $d(G)$ from $(s,P)$ as follows. If $s=-$, then $d(G)$ is the number of elements in the even-numbered blocks of $P$ (the second, fourth, et cetera), while if $s=+$ it is one less than the number of elements in the odd-numbered blocks of $P$ (the ``one less'' because we do not consider the initial vertex in the iterative construction to be dominating).  

It follows that there are four templates for a labelled threshold graph on vertex set $[n]$ with $k$ dominating vertices in any iterative construction of $G$. 
\begin{enumerate}[(i)]
\item $(+,P)$, with an odd number of blocks in $P$, and with $k+1$ elements in odd-numbered blocks. To enumerate these pairs, we have to choose:
\begin{itemize}
\item $\ell \geq 1$, the number of odd-numbered blocks;
\item the elements in the union of those blocks --- $\binom{n}{k+1}$ options; and
\item the ordered block decomposition among these $\ell$ blocks, noting that the first block must have at least two elements --- $\op2_{k+1,\ell}$ options;
\item the ordered block decomposition of the $\ell-1$ blocks that get alternately interspersed among the $\ell$ already chosen blocks, noting that there are no restrictions on the block sizes --- $(\ell-1)!\stirling2{n-k-1}{\ell-1}$ options.  
\end{itemize}
So the number of pairs following this template is
$$
\sum_{\ell \geq 1} \binom{n}{k+1}\op2_{k+1,\ell}(\ell-1)!\stirling2{n-k-1}{\ell-1}.
$$
\item $(+,P)$, with an even number of blocks in $P$, with $k+1$ elements in odd-numbered blocks. By similar reasoning to case (i) above, the number of pairs following this template is
$$
\sum_{\ell \geq 1} \binom{n}{k+1}\op2_{k+1,\ell}\ell!\stirling2{n-k-1}{\ell}.
$$
\item $(-,P)$, with an odd number of blocks in $P$, with $k$ elements in even-numbered blocks. The number of pairs following this template is
$$
\sum_{\ell \geq 1} \binom{n}{k}\ell!\stirling2{k}{\ell}\op2_{n-k,\ell+1}.
$$
\item $(-,P)$, with an even number of blocks in $P$, with $k$ elements in even-numbered blocks. The enumeration of pairs following this template is
$$
\sum_{\ell \geq 1} \binom{n}{k}\ell!\stirling2{k}{\ell}\op2_{n-k,\ell}.
$$
\end{enumerate}
Combining (i) through (iv), \eqref{eq-dom-thresh'} follows.

\section{Proofs for quasi-threshold graph results} \label{sec-quasi-thresh-proofs}

\subsection{Proof of \eqref{qtn1}} \label{sec-qtn1-proof}

The goal of this section is to obtain a combinatorial proof of the identity
\begin{equation} \label{qtn1'}
\qt_n = \sum_{k=0}^n (-1)^{n-k} \stirling2{n}{k}(k+1)^{k-1},
\end{equation}
where $\qt_n$ is the number of labelled quasi-threshold graphs on $[n]$. The proof will be via a sign-changing involution; see the discussion just after \eqref{qtn1}.

Say that a rooted forest $F$ is a {\it rooted partition forest} on a finite set $A \subseteq {\mathbb N}$ if the vertices of $F$ are the blocks of a partition of $A$ (as usual, into non-empty blocks). Let ${\mathcal F}_A$ be the set of rooted partition forests on $A$. Let ${\mathcal F}^{\mathcal O}_A$ be the set of forests $F$ in ${\mathcal F}_A$ with $|A|-k$ odd, where $k$ is the number of vertices of $F$, and let ${\mathcal F}^{\mathcal E}_A$ be the set of those with $|A|-k$ even.

Write ${\mathcal F}_n$ as shorthand for ${\mathcal F}_{[n]}$. Since the number of rooted forests on $k$ labelled vertices is $(k+1)^{k-1}$ (this is essentially Cayley's formula), we have
\begin{equation} \label{eq-invol}
\begin{array}{rcl}
|{\mathcal F}_n| & = & \sum_{k=0}^n \stirling2{n}{k}(k+1)^{k-1}~\mbox{and} \\
|{\mathcal F}^{\mathcal E}_n|-|{\mathcal F}^{\mathcal O}_n| & = & \sum_{k=0}^n (-1)^{n-k}\stirling2{n}{k}(k+1)^{k-1}.
\end{array}
\end{equation}

We will prove \eqref{qtn1'} in two steps.
\begin{enumerate}[(a)]
\item First, we will identify a subset ${\mathcal Q}_n$ of ${\mathcal F}_n$ that is equinumerous with the set of labelled quasi-threshold graphs on $[n]$. All $F \in {\mathcal Q}_n$ will have $n$ vertices (so, the underlying partition of $[n]$ that forms the vertex set of $F$ will be the partition into $n$ singleton blocks), and so be in ${\mathcal F}^{\mathcal E}_n$.
\item Then we will describe an involution $\iota$ on ${\mathcal F}_n$ with the following properties:
\begin{enumerate}[(i)]
\item if $\{F,F'\}$ is an orbit of size two of $\iota$, then one of $F, F'$ is in ${\mathcal F}^{\mathcal O}_n$ and the other is in ${\mathcal F}^{\mathcal E}_n$, and
\item the set of fixed points of $\iota$ is ${\mathcal Q}_n$.
\end{enumerate}
\end{enumerate}
Via \eqref{eq-invol}, this establishes \eqref{qtn1'}.

We start with (a). We will first encode labelled quasi-threshold graphs on $[n]$ by rooted partition forests on $[n]$, using an encoding implicitly given in \cite[Section 2.1]{NikolopouosPapadopolous}. We then modify the rooted partition forests on $[n]$ that occur in this encoding to create the set ${\mathcal Q}_n$. The initial encoding of labelled quasi-threshold graphs by rooted partition forests is obtained inductively as follows.
\begin{itemize}
\item
If a labelled quasi-threshold graph $G$ on $[n]$ has $\ell$ components, then the associated rooted forest $F(G)$ has $\ell$ components, one corresponding to each component of $G$; and the component of $F(G)$ corresponding to a particular component $C$ is a rooted partition tree on the vertex set of $C$.
\item
If a component $C$ of $G$ is a singleton, then the associated rooted partition tree is a single vertex; more generally, if $C$ is a complete graph, then the associated rooted partition tree is a single vertex, whose label is the vertex set of $C$.
\item
If on the other hand $C$ is not a singleton and is not complete, then it has a collection $D$ of dominating vertices, and the graph obtained from $C$, by deleting $D$, has more than one component. These components, say $C_1, \ldots, C_\ell$, inductively have associated rooted partition trees, $T(C_1), \ldots, T(C_\ell)$, say. The rooted partition tree associated with $C$ is obtained from $T(C_1), \ldots, T(C_\ell)$ by adding a root with label $D$, and an edge from $D$ to the root of $T(C_i)$ for each $i$.
\end{itemize}
(See Figures \ref{fig-rpt1} and \ref{fig-rpt2}.)
\begin{figure}[ht!] 
\begin{center}
\begin{tikzpicture}[scale=0.93]
\node[label=below:5] at (0,0) [circle,fill,inner sep=1.5pt]{};
\node[label=below:1] at (1,0) [circle,fill,inner sep=1.5pt]{};
\node[label=below:9] at (2,.5) [circle,fill,inner sep=1.5pt]{};
\node[label=below:13] at (3,1.5) [circle,fill,inner sep=1.5pt]{};
\node[label=below:10] at (4,3) [circle,fill,inner sep=1.5pt]{};
\node[label=above:17] at (4,4) [circle,fill,inner sep=1.5pt]{};
\node[label=above:16] at (3,4) [circle,fill,inner sep=1.5pt]{};
\node[label=below:7] at (5,-1) [circle,fill,inner sep=1.5pt]{};
\node[label=below:15] at (6,-1) [circle,fill,inner sep=1.5pt]{};
\node[label=below:11] at (7,-.5) [circle,fill,inner sep=1.5pt]{};
\node[label=below right:14] at (8,.5) [circle,fill,inner sep=1.5pt]{}; 
\node[label=below:3] at (9,2) [circle,fill,inner sep=1.5pt]{};
\node[label=below:4] at (10,0) [circle,fill,inner sep=1.5pt]{};
\node[label=below:2] at (11,0) [circle,fill,inner sep=1.5pt]{};
\node[label=below:6] at (12,.5) [circle,fill,inner sep=1.5pt]{};
\node[label=above:8] at (13,4) [circle,fill,inner sep=1.5pt]{};
\node[label=above:12] at (14,4) [circle,fill,inner sep=1.5pt]{};
\draw (4,4) -- (3,4);
\draw (4,4) -- (0,0);
\draw (4,4) -- (1,0);
\draw (4,4) -- (2,.5);
\draw (4,4) -- (3,1.5);
\draw (4,4) -- (4,3);
\draw (3,4) -- (0,0);
\draw (3,4) -- (1,0);
\draw (3,4) -- (2,.5);
\draw (3,4) -- (3,1.5);
\draw (3,4) -- (4,3);
\draw (1,0) -- (0,0);
\draw (4,3) -- (0,0);
\draw (4,3) -- (1,0);
\draw (4,3) -- (2,.5);
\draw (4,3) -- (3,1.5);
\draw (8,.5) -- (5,-1);
\draw (8,.5) -- (6,-1);
\draw (8,.5) -- (7,-.5);
\draw (9,2) -- (5,-1);
\draw (9,2) -- (6,-1);
\draw (9,2) -- (7,-.5);
\draw (9,2) -- (8,.5);
\draw (10,0) -- (11,0);
\draw (13,4) -- (5,-1);
\draw (13,4) -- (6,-1);
\draw (13,4) -- (7,-.5);
\draw (13,4) -- (8,.5);
\draw (13,4) -- (9,2);
\draw (13,4) -- (10,0);
\draw (13,4) -- (11,0);
\draw (13,4) -- (12,.5);
\draw (14,4) -- (5,-1);
\draw (14,4) -- (6,-1);
\draw (14,4) -- (7,-.5);
\draw (14,4) -- (8,.5);
\draw (14,4) -- (9,2);
\draw (14,4) -- (10,0);
\draw (14,4) -- (11,0);
\draw (14,4) -- (12,.5);
\draw (13,4) -- (14,4);
\end{tikzpicture}
\end{center}
\caption{A quasi-threshold graph. The component on the left is a threshold graph with code $(+,5~1/9~13/10~16~17)$, and the component on the right has two dominating vertices (8 and 12) joined to the union of two threshold graphs; one with code $(-,7~15~11/3~14)$ and the other with code $(+,2~4/6)$.}
\label{fig-rpt1}
\end{figure}
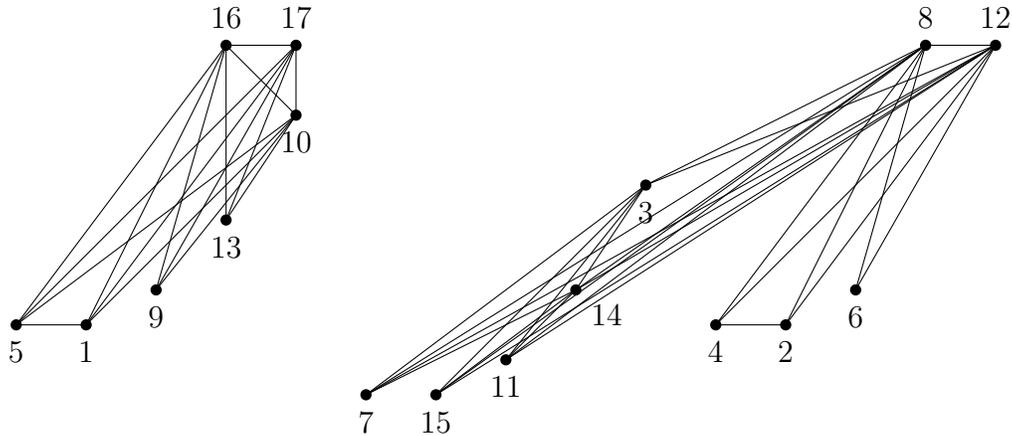

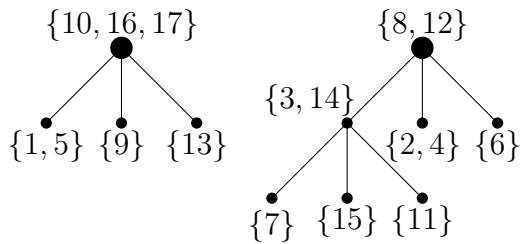
\begin{figure}[ht!]
\begin{center}
\begin{tikzpicture}
\node at (0,0) [circle,fill,inner sep=3pt]{};
\node at (0,.3) {$\{10,16,17\}$};
\node at (-1,-1) [circle,fill,inner sep=1.5pt]{};
\node at (-1,-1.3) {$\{1,5\}$};
\node at (0,-1) [circle,fill,inner sep=1.5pt]{};
\node at (0,-1.3) {$\{9\}$};
\node at (1,-1) [circle,fill,inner sep=1.5pt]{};
\node at (1,-1.3) {$\{13\}$};
\draw (0,0) -- (-1,-1);
\draw (0,0) -- (0,-1);
\draw (0,0) -- (1,-1);
\node at (4,0) [circle,fill,inner sep=3pt]{};
\node at (4,.3) {$\{8,12\}$};
\node at (3,-1) [circle,fill,inner sep=1.5pt]{};
\node at (2.5,-.7) {$\{3,14\}$};
\node at (4,-1) [circle,fill,inner sep=1.5pt]{};
\node at (4,-1.3) {$\{2,4\}$};
\node at (5,-1) [circle,fill,inner sep=1.5pt]{};
\node at (5,-1.3) {$\{6\}$};
\draw (4,0) -- (3,-1);
\draw (4,0) -- (4,-1);
\draw (4,0) -- (5,-1);
\node at (2,-2) [circle,fill,inner sep=1.5pt]{};
\node at (2,-2.37) {$\{7\}$};
\node at (3,-2) [circle,fill,inner sep=1.5pt]{};
\node at (3,-2.3) {$\{15\}$};
\node at (4,-2) [circle,fill,inner sep=1.5pt]{};
\node at (4,-2.3) {$\{11\}$};
\draw (3,-1) -- (2,-2);
\draw (3,-1) -- (3,-2);
\draw (3,-1) -- (4,-2);
\end{tikzpicture}
\end{center}
\caption{The rooted partition forest associated with the quasi-threshold graph in Figure \ref{fig-rpt1}.}
\label{fig-rpt2}
\end{figure}

The set of rooted partition forests on $[n]$ that arise in this encoding, as we vary over labelled quasi-threshold graphs $G$ on $[n]$, is precisely the set of those in which all non-leaf vertices have degree at least $2$ (these are sometimes referred to as {\it phylogenetic forests}). These forests will typically have some vertices whose labels are blocks of size 2 or greater; for the purposes of obtaining a sign-changing involution (step (b)) it is helpful to make a minor modification, so that all the blocks are singletons.  

Specifically, in each rooted partition forest on $[n]$ that arises in the encoding, we replace each vertex with a path; if the vertex has label $C$, the path has $|C|$ vertices, and the labels on the vertices of the path are precisely the elements of $C$, arranged so that they are encountered in increasing order as we move away from the root (of whatever component we happen to be in). (See Figure \ref{fig-rpt3}.)   

\begin{figure}[ht!]
\begin{center}
\begin{tikzpicture}
\node[label=above:10] at (0,0) [circle,fill,inner sep=3pt]{}; 
\node[label=right:16] at (0,-1) [circle,fill,inner sep=1.5pt]{}; 
\node[label=above right:17] at (0,-2) [circle,fill,inner sep=1.5pt]{}; 
\node[label=above left:1] at (-1,-3) [circle,fill,inner sep=1.5pt]{}; 
\node[label=below:9] at (0,-3) [circle,fill,inner sep=1.5pt]{}; 
\node[label=below:13] at (1,-3) [circle,fill,inner sep=1.5pt]{}; 
\node[label=below:5] at (-1,-4) [circle,fill,inner sep=1.5pt]{}; 
\node[label=above:8] at (5,0) [circle,fill,inner sep=3pt]{}; 
\node[label=above right:12] at (5,-1) [circle,fill,inner sep=1.5pt]{}; 
\node[label=above:3] at (4,-2) [circle,fill,inner sep=1.5pt]{}; 
\node[label=above right:14] at (4,-3) [circle,fill,inner sep=1.5pt]{}; 
\node[label=above right:2] at (5,-2) [circle,fill,inner sep=1.5pt]{}; 
\node[label=below:4] at (5,-3) [circle,fill,inner sep=1.5pt]{}; 
\node[label=below:6] at (6,-2) [circle,fill,inner sep=1.5pt]{};
\node[label=below:7] at (3,-4) [circle,fill,inner sep=1.5pt]{}; 
\node[label=below:11] at (4,-4) [circle,fill,inner sep=1.5pt]{}; 
\node[label=below:15] at (5,-4) [circle,fill,inner sep=1.5pt]{}; 
\draw[dashed] (0,0) -- (0,-1);
\draw[dashed] (0,-1) -- (0,-2);
\draw (0,-2) -- (-1,-3);
\draw (0,-2) -- (0,-3);
\draw (0,-2) -- (1,-3);
\draw (-1,-3) -- (-1,-4);
\draw (5,0) -- (5,-1);
\draw (5,-1) -- (4,-2);
\draw (5,-1) -- (5,-2);
\draw (5,-1) -- (6,-2);
\draw (4,-2) -- (4,-3);
\draw (5,-2) -- (5,-3);
\draw (4,-3) -- (3,-4);
\draw (4,-3) -- (4,-4);
\draw (4,-3) -- (5,-4);
\end{tikzpicture}
\end{center}
\caption{The modification of the rooted partition forest from Figure \ref{fig-rpt2} to form a quasi-threshold forest $F$. For typographic convenience we have not put the labels in braces here. The dashed edges indicate $h(F)$.}
\label{fig-rpt3}
\end{figure}
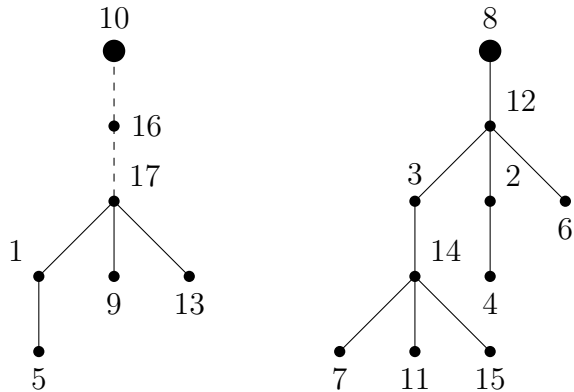

We refer to the forests produced by this modification as {\it quasi-threshold forests}, and let ${\mathcal Q}_n$ be the set of all such forests; note that $|{\mathcal Q}_n|=\qt_n$, and that each $F \in {\mathcal Q}_n$ has $n$ vertices, and so as promised is in ${\mathcal F}^{\mathcal E}_n$. 

Next we turn to (b). We describe an involution $\iota$ defined on ${\mathcal F}_n$ that fixes ${\mathcal Q}_n$, and that changes the parity of the number of vertices for each rooted partition forest that is not in ${\mathcal Q}_n$.

We describe the involution recursively. Let $F$, a rooted partition forest on $A$, be given, where $A$ is a finite subset of the natural numbers. (Our interest is in $A=[n]$, but for the purposes of describing the involution recursively, it is more convenient to generalize slightly). The construction of $\iota$ involves examining a particular subgraph $h(F)$ of $F$, and, depending on the nature of that subgraph, either making a modification to $F$ to produce $\iota(F)$, or deleting the subgraph and recursing.  

We begin by describing the process by which $h(F)$ it identified.
\begin{itemize}
\item 
Let $F'$ be the component of $F$ that contains among its labels the minimum element of $A$.
\item
If $F'$ consists of a single vertex, then we take $h(F)$ to be $F'$. If $F'$ has more than one vertex, then let $a$ be the root of $F'$, and let $b$ be the first vertex of $F'$, encountered while travelling away from $a$, that is either a leaf or has more than one child (note that $b$ may be the root itself). Equivalently, $b$ is the last vertex on the longest path in $F'$ that starts at $a$ and that, other than $a$ and $b$, only uses vertices of degree $2$. Let $h(F)$ be this path from $a$ to $b$.
\end{itemize}
(See Figure \ref{fig-rpt3}.)

We can now describe the process that produces $\iota(F)$. It will have two parts: if $h(F)$ has all singleton labels that occur in increasing order (consistent with $F$ being a quasi-threshold forest), then we recurse; otherwise, we immediately make a modification to produce $\iota(F)$. In the description that follows, we treat this second situation first.  
\begin{itemize}
\item If it is not the case that $h(F)$ has all singleton labels that occur in increasing order, then let $v$ be the first vertex along the path from $a$ to $b$ that witnesses this failure. There are two (mutually exclusive) possibilities for $v$:
\begin{enumerate}[(A)]
\item The label of $v$ is not a singleton, and either
\begin{itemize}
\item $v=a$, or
\item the path from $a$ up to (but not including) $v$ has all singleton labels that occur in increasing order, and at least one of the elements in the label of $v$ is larger than the label at the parent of $v$.
\end{itemize} 
In this case, replace $v$ with two vertices, $v'$ and $v''$, with $v''$ a child of $v'$, $v'$ a child of the parent of $v$ (if $v$ has a parent; otherwise, $v'$ becomes a root), and $v''$ the parent of all the children of $v$. Set the label of $v'$ to be the largest element in the label of $v$, and set the label of $v''$ to be all other elements in the label of $v$. Let the resulting rooted partition forest on $A$ be $\iota(F)$.
\item $v\neq a$, the path from $a$ up to (but not including) $v$ has all singleton labels that occur in increasing order, and all the elements in the label of $v$ (which now may or may not be a singleton) are smaller than the unique element in the label of the parent of $v$. In this case, contract (in the usual graph-theoretic sense) along the edge joining $v$ and its parent, and set the label of the newly created vertex $\widetilde{v}$ to be the union of the labels of $v$ and its parent. Let the resulting rooted partition forest on $A$ be $\iota(F)$.
\end{enumerate}
(See Figure \ref{fig-AB1}.)

\begin{figure}[ht!] 
\begin{center}
\begin{tikzpicture}
\node at (0,0) [circle,fill,inner sep=1.5pt]{};
\node at (0,.3) {$a=\{4\}$};
\node at (0,-1) [circle,fill,inner sep=1.5pt]{};
\node at (0.5,-1) {$\{6\}$};
\node at (0,-2) [circle,fill,inner sep=1.5pt]{};
\node at (1.5,-2) {$v=\{2,5,7,9\}$};
\draw (0,0) -- (0,-1);
\draw (0,-1) -- (0,-2);
\draw[dashed] (0,-2) -- (-.7,-2.7);
\node at (7,0) [circle,fill,inner sep=1.5pt]{};
\node at (7,.3) {$a=\{4\}$};
\node at (7,-1) [circle,fill,inner sep=1.5pt]{};
\node at (7.5,-1) {$\{6\}$};
\node at (7,-2) [circle,fill,inner sep=1.5pt]{};
\node at (8,-2) {$v'=\{9\}$};
\node at (7,-3) [circle,fill,inner sep=1.5pt]{};
\node at (8.4,-3) {$v''=\{2,5,7\}$};
\draw (7,0) -- (7,-1);
\draw (7,-1) -- (7,-2);
\draw (7,-2) -- (7,-3);
\draw[dashed] (7,-3) -- (6.3,-3.7);
\draw[dashed,thick,<->] (2,-1) -- (6,-1);
\node at (4,-.8) {$\iota$};
\end{tikzpicture}
\end{center}
\caption{The vertex $v$ in the graph on the left is of type (A), and the graph on the right shows the result after modifying at $v$. The vertex $v'$ in the graph on the right is of type (B), and the graph on the left shows the result after modifying at $v'$.}
\label{fig-AB1}
\end{figure}
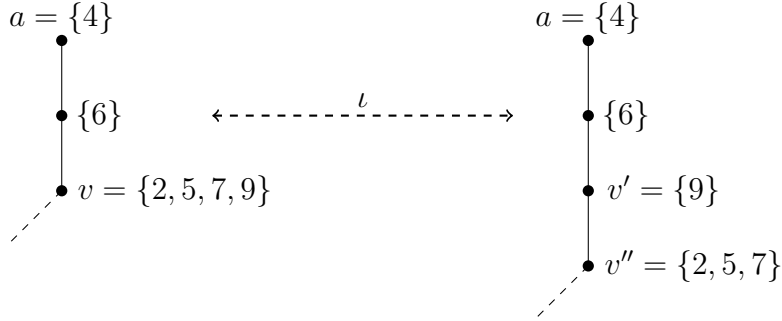

\item If $h(F)$ has all singleton labels that occur in increasing order, then delete $h(F)$ from $F$. We now perform one of two steps: 
\begin{enumerate}[(1)]
\item If the result is the empty graph, then the process stops, and $\iota(F)=F$
\item Otherwise, the resulting graph is a rooted partition forest $\widetilde{F}$ on a proper non-empty subset $B$ of $A$ (the children of $b$, if any, become roots of their components). In this case, on $\widetilde{F}$ we run the algorithm currently being described, beginning with identifying $h(\widetilde{F})$. This will (eventually, recursively) produce a rooted partition forest $\iota(\widetilde{F})$ on $B$. We turn this into a rooted partition forest on $A$ by restoring $h(F)$ (with $a$ as root), and joining $b$ to all the roots of the components in $\iota(\widetilde{F})$ that correspond to components in $F'-h(F)$ (from the construction of $\iota$ there is a natural one-to-one correspondence between these two sets of components). Set $\iota(F)$ to be the rooted partition forest on $[n]$ thus produced from $\widetilde{F}$.
\end{enumerate}
\end{itemize}

Having defined $\iota$, we now make the observations necessary to show that when $A=[n]$, $\iota$ satisfies all conditions required by (b) (which recall are that $\iota$ is an involution on ${\mathcal F}_n$, that if $\{F,F'\}$ is an orbit of size two of $\iota$ then one of $F, F'$ is in ${\mathcal F}^{\mathcal O}_n$ and the other is in ${\mathcal F}^{\mathcal E}_n$, and that the set of fixed points of $\iota$ is ${\mathcal Q}_n$).

First, note that if $F$ is a quasi-threshold forest on $[n]$ then by construction we have $\iota(F)=F$, and conversely if $F$ is not a quasi-threshold forest on $[n]$ then that failure will be witnessed at some vertex, and so $\iota(F) \neq F$. So the set of fixed points of $\iota$ is certainly ${\mathcal Q}_n$ 

Second, note that if $\iota(F) \neq F$, then the number of vertices in $F$ and $F'$ have different parities (when $\iota$ makes a change to a rooted partition forest, it either adds or subtracts a vertex).

So all that remains in the verification of (b) (and thus \eqref{qtn1'}) is to argue that $\iota$ is an involution, which we now do.

\begin{claim}
$\iota$ is an involution on ${\mathcal F}_n$.
\end{claim}

\begin{proof}
Suppose first that $\iota(F)$ is obtained from $F$ by an application of the modification described in (A). Whether or not $v=a$, in $\iota(F)$ the first vertex that witnesses the failure of $\iota(F)$ to be a rooted partition forest is $v''$, and $\iota(\iota(F))$ is obtained from $\iota(F)$ by an application of the modification described in (B); this undoes the change that turned $F$ into $\iota(F)$, so in this case $\iota(\iota(F))=F$.

Suppose on the other hand that $\iota(F)$ is obtained from $F$ by an application of the modification described in (B). Then in $\iota(F)$ the first vertex that witnesses the failure of $\iota(F)$ to be a rooted partition forest is $\widetilde{v}$, and $\iota(\iota(F))$ is obtained from $\iota(F)$ by an application of the modification described in (A); this undoes the change that turned $F$ into $\iota(F)$, so in this case also $\iota(\iota(F))=F$.
\end{proof}

\subsection{Proof of \eqref{compqtn}} \label{sec-qt-comp-proofs}

The goal of this section is to establish that 
\begin{equation} \label{compqtn'}
\qtcomp_{n,\ell} = \sum_{k=1}^n (-1)^{n-k} \stirling2{n}{k}\ell\binom{k}{\ell}k^{k-\ell-1},
\end{equation}
where $\qtcomp_{n,\ell}$ is the number of labelled quasi-threshold graphs on $n$ vertices with $\ell$ components (generalizing \eqref{connqtn1}, the case $\ell=1$). 

This follows immediately from the involution given in Section \ref{sec-qtn1-proof}. Indeed, let $f(k,\ell)$ denote the number of rooted forests on $k$ vertices with $\ell$ components; so $f(k,1)=k^{k-1}$, and
$$
f(k,\ell) = \ell\binom{k}{\ell}k^{k-\ell-1}
$$
(see e.g. \cite[Chapter 33]{AignerZiegler}). It follows that $\sum_{k=1}^n \stirling2{n}{k}\ell\binom{k}{\ell}k^{k-\ell-1}$ enumerates rooted partition forests on $[n]$ that have $k$ components. Restricting $\iota$ to this set, and noting that $\iota$ does not change the number of components of a rooted partition forest, the discussion in Section \ref{sec-qtn1-proof} shows that the right-hand side of \eqref{compqtn'} counts quasi-threshold forests on $[n]$ with $\ell$ components. By construction, the set of such forests is equinumerous with the set of labelled quasi-threshold graphs on $[n]$ with $\ell$ components.    

\section{Proofs for loop-threshold graph results} \label{sec-loop-thresh-proofs}

\subsection{Proof of \eqref{eq-lt-basic}} \label{sec-lt-initial-proofs}

The goal of this section is to give a combinatorial proof of the formula 
\begin{equation} \label{eq-lt-basic'}
\lt_n = 2\sum_{k=1}^n k!\stirling2{n}{k},
\end{equation}
where $\lt_n$ is the number of labelled loop-threshold graphs on $[n]$.

Recall that in Section \ref{sec-thresh-proofs} we observed that a labelled threshold graph on $[n]$ can be encoded by a pair $(s, P)$ where $s \in \{+,-\}$ and $P$ is an ordered partition of $[n]$ with the first block having size at least $2$.

By almost the same process, a labelled loop-threshold graph on $[n]$ can be encoded by a pair $(s, P)$ where $s \in \{+,-\}$ and $P$ is an ordered partition of $[n]$, with no restriction on the size of the first block. We now describe this encoding.

Recall that a loop-threshold graph $G$ can be constructed by either starting with an unlooped or a looped vertex, and then iteratively adding isolated or looped dominating vertices. If the construction of $G$ begins with a looped vertex, then take $s=+$, and take as the first block of $P$ the label assigned to the initial loop, together with the labels of all the added looped dominating vertices that are added before the first isolated vertex is added; then take the second block of $P$ to be the labels of all the isolated vertices that are added before the next looped dominating vertex is added, and so on. If the construction of $G$ begins with an isolated vertex, take $s=-$ and then proceed similarly. It is easily seen that each loop-threshold graph has a unique code, and that for each pair $(s, P)$ there is a unique loop-threshold graph that has $(s, P)$ as its code. (See Figure \ref{fig-loop-thresh-codes}.) 

\begin{figure}[ht!] 
\begin{center}
\begin{tikzpicture}
\node[label=below:2] at (0,0) [circle,fill,inner sep=3.5pt]{};
\node[label=below:3] at (1,0) [circle,fill,inner sep=3.5pt]{};
\node[label=below:1] at (2,.5) [circle,fill,inner sep=1.5pt]{};
\node[label=below:4] at (3,1.5) [circle,fill,inner sep=1.5pt]{};
\node[label=below:5] at (4,3) [circle,fill,inner sep=3.5pt]{};
\node[label=below:1] at (7,0) [circle,fill,inner sep=1.5pt]{};
\node[label=below:3] at (8,0) [circle,fill,inner sep=3.5pt]{};
\node[label=below:5] at (9,.5) [circle,fill,inner sep=1.5pt]{};
\node[label=below:2] at (10,1.5) [circle,fill,inner sep=3.5pt]{};
\node[label=below:4] at (11,3) [circle,fill,inner sep=3.5pt]{};
\draw (1,0) -- (0,0);
\draw (4,3) -- (0,0);
\draw (4,3) -- (1,0);
\draw (4,3) -- (2,.5);
\draw (4,3) -- (3,1.5);
\draw (10,1.5) -- (7,0);
\draw (10,1.5) -- (8,0);
\draw (10,1.5) -- (9,.5);
\draw (11,3) -- (7,0);
\draw (11,3) -- (8,0);
\draw (11,3) -- (9,.5);
\draw (11,3) -- (10,1.5);
\draw (7,0) -- (8,0);
\end{tikzpicture}
\end{center}
\caption{Two labelled loop threshold graphs. The one on the left has code $(+,23/14/5)$, while the one on the right has code $(-,1/3/5/24)$. (Larger nodes are looped, smaller nodes are unlooped.)}
\label{fig-loop-thresh-codes}
\end{figure}
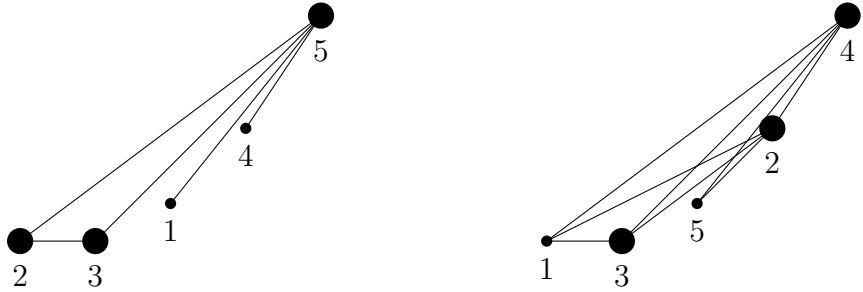

Since the number of ordered partitions of $[n]$ is $\sum_{k=1}^n k!\stirling2{n}{k}$, \eqref{eq-lt-basic'} follows.

\subsection{Proof of \eqref{ltn2-conj}} \label{sec-lt-Frob}

To goal of this section is to give a combinatorial proof of the identity 
\begin{equation} \label{ltn2-conj'}
\lt_n = \sum_{k=0}^{n-1} \eulerian{n}{k}2^{k+1},
\end{equation}
and also to extend it to a combinatorial proof of the Frobenius identity \eqref{frobenius}.

In Section \ref{sec-lt-initial-proofs} we observed that labelled loop-threshold graphs on vertex set $[n]$ may be encoded by pairs $(s, P)$ where $s \in \{+,-\}$ and $P$ is an ordered partition of $[n]$. Since the right-hand side of \eqref{ltn2-conj'} enumerates triples $(s,\pi, c)$ where $s \in \{+,-\}$, $\pi$ is a permutation of $[n]$, and $c$ is a $2$-coloring of the ascents of $\pi$, to establish \eqref{ltn2-conj'} it suffices to establish a bijection between the set of ordered partitions of $[n]$, and the set of permutations of $[n]$ together with a $2$-coloring of the ascents of the permutation.

Here is one such bijection: given an ordered partition $P$ of $[n]$, set $\pi=\pi_f(P)$ (see Definition \ref{def-flat}). This permutation has some ascents $\pi_i$ with the property that $\pi_i$ and $\pi_{i+1}$ are both in the same block of $P$. Color these ascents red, and color the remaining ascents blue. This is easily seen to an invertible map. Indeed, given a permutation with a $2$-coloring of its ascents, the ordered partition that it comes from can be obtained by breaking the permutation at descents, and at ascents colored blue.  

This argument easily extends to a combinatorial proof of the Frobenius identity, which recall says
\begin{equation} \label{frobenius'}
\sum_{k=0}^{n-1} \eulerian{n}{k}x^k = \sum_{\ell=1}^n \ell!\stirling2{n}{\ell}(x-1)^{n-\ell}.
\end{equation}
When $x$ is a positive integer, the right-hand side of \eqref{frobenius} counts
\begin{quote}
ordered partitions of $[n]$, together with a coloring (from a palette $[x-1]$ of $x-1$ colors) of the elements of $[n]$ that are not the largest entry in their blocks, 
\end{quote}
while the left-hand side counts
\begin{quote}
permutations of $[n]$, together with a coloring (from a palette $[x]$ of $x$ colors) of the ascents of the permutation. 
\end{quote}
Here is a bijection from the first set to the second. Given an ordered partition $P$ of $[n]$, together with an $(x-1)$-coloring $c$ of the elements that are not largest in their blocks, set $\pi=\pi_f(P)$ (see Definition \ref{def-flat}). All elements of $[n]$ that are not largest in their blocks are ascents in $\pi$, so the coloring $c$ is a partial $x$-coloring of the ascents of $\pi$. Extend it to a full $x$-coloring by giving all remaining ascents color $x$. This map is easily seen to be invertible. Indeed, given a permutation with an $x$-coloring of its ascents, the ordered partition that it comes from can be obtained by breaking the permutation at all descents, and at ascents given color $x$.

\subsection{Proof of \eqref{ltn1-conj}} \label{sec-lt-involution}

The goal of this section is to give a combinatorial proof of the identity 
\begin{equation} \label{ltn1-conj'}
\lt_n = \sum_{k=1}^n (-1)^{n-k}\stirling2{n}{k}k!2^k
\end{equation}
for $n \geq 0$. We have dropped the $k=0$ from \eqref{ltn1-conj} here, since $\stirling2{n}{0} = 1$ for $n \geq 1$; note that \eqref{ltn1-conj} holds vacuously for $n=0$.
The proof of \eqref{ltn1-conj'} will be via a sign-changing involution; see the discussion just after \eqref{qtn1}.

We begin by observing that $\displaystyle\sum_{k=1}^n \stirling2{n}{k}k!2^k$ counts pairs $(P,c)$ where $P$ is an ordered partition of $[n]$ and $c:\{P_1,\ldots,P_k\} \rightarrow \{{\rm red},{\rm blue}\}$ is a $2$-coloring of the blocks of $P$.  

Let ${\mathcal P}_n$ be the collection of all such pairs, and let ${\mathcal P}_{n,k}$ be those pairs in which $P$ has $k$ blocks (so our goal is to establish $\lt_n = \sum_{k=1}^n (-1)^{n-k}|{\mathcal P}_{n,k}|$). Let ${\mathcal L}_n \subseteq {\mathcal P}_n$ be the set of pairs $(P,c)$ with $P=P_1/\cdots/P_k$ with the following two properties:
\begin{itemize}
\item each block $P_i$ in $P$ is a singleton, say $P_i=\{p_i\}$, (so $k=n$ and ${\mathcal L}_n \subseteq {\mathcal P}_{n,n}$) and
\item if $c(P_i)=c(P_{i+1})$ (that is, if two consecutive blocks have the same color) then $p_i < p_j$. 
\end{itemize}

\begin{claim} \label{clm-lltg}
$\lt_n = |{\mathcal L}_n|$.
\end{claim}

\begin{proof}
Recall from Section \ref{sec-lt-initial-proofs} that
a labelled loop-threshold graph on $[n]$ can be encoded by a pair $(s, P)$ where $s \in \{+,-\}$ and $P$ is an ordered partition of $[n]$. Given such a pair, we modify it to turn it into an element of ${\mathcal L}_n$ as follows. Turn $P=P_1/\cdots/P_k$ into an ordered partition $P'$ consisting of singleton blocks, by putting every element of block $i$ before every element of block $j$ for each $i < j$, and within block $i$, putting the elements in increasing order. If $s=+$, then let $c(a)$ be red if $a$ is in block $P_i$ of $P$ for some odd $i$, and blue otherwise; if $s=-$ then flip the roles of red and blue. This gives a map from labelled loop-threshold graph on $[n]$ to ${\mathcal L}_n$ that is clearly invertible. 
\end{proof}

\begin{example} \label{ex-loopbij}
The pair $(+,23/14/5)$ maps to the pair  $(\textcolor{red}{2}/\textcolor{red}{3}/\textcolor{blue}{1}/\textcolor{blue}{4}/\textcolor{red}{5}, c)$ (a coloured ordered partition) where $c^{-1}({\rm red})=\{2,3,5\}$ and $c^{-1}({\rm blue})=\{1,4\}$, while $(-,1/3/5/24)$ maps to $(\textcolor{blue}{1}/\textcolor{red}{3}/\textcolor{blue}{5}/\textcolor{red}{2}/\textcolor{red}{4},c)$ where $c^{-1}({\rm red})=\{2,3,4\}$ and $c^{-1}({\rm blue})=\{1,5\}$.
\end{example}

In the presence of Claim \ref{clm-lltg}, to prove \eqref{ltn1-conj'} it suffices to find an involution $\iota$ on ${\mathcal P}_n$ with the following two properties:
\begin{enumerate}[(i)]
\item if $\{(P,c),(P',c')\}$ is an orbit of size two of $\iota$, then one of $(P,c), (P',c')$ is in ${\mathcal P}_{n,k}$ with $k$ even, and the other is in ${\mathcal P}_{n,k}$ with $k$ odd (that is, $\iota$ changes the parity of the number of blocks in $P$), and
\item The set of fixed points of $\iota$ is ${\mathcal L}_{n}$.
\end{enumerate}

We first describe such an $\iota$, and then show that it has all the claimed properties.

Given an ordered partition $P=P_1/\cdots/P_k$ of $[n]$ and a $2$-coloring $c$ of the blocks of $P$, scan the blocks in increasing order ($P_1$ first, then $P_2$, et cetera), until the first time that a block $P_i$ is encountered that satisfies one of the following two (mutually exclusive) conditions:
\begin{enumerate}[(A)]
\item $P_i$ is not a singleton, or
\item $P_i=\{p_i\}$ is a singleton, and furthermore satisfies
\begin{itemize}
\item $i < k$ (that is, $P_i$ is not the final block of $P$),
\item $c(P_i) = c(P_{i+1})$, and 
\item $p_i > p$ for all $p \in P_{j+1}$ (that is, the unique element in $P_i$ is larger than everything in the block that comes after it).
\end{itemize}
\end{enumerate}

If the first such block that is encountered is $P_i$, of type (A), with $P_i=\{p_{i1}, \ldots, p_{i\ell}\}$ ($p_{i1} <\cdots < p_{i\ell}$, $\ell \geq 2$), then modify $(P,c)$ as follows.
\begin{itemize}
\item Replace $P$ with $P'=P_1/\cdots/P_{i-1}/P_i'/P_i''/P_{i+1}/\cdots/P_k$ where $P_i'=p_{i\ell}$ and $P_i''=\{p_{i1}, \ldots, p_{i(\ell-1)}\}$, and
\item replace $c$ with $c'$ that agrees with $c$ on $P_j$, $j \neq i$, and also satisfies $c'(P_i')=c'(P_i'')=c(P_i)$.
\end{itemize}
(See Example \ref{ex-flip-bij} below.)
In this case, set $\iota(P,c)=(P',c')$.

If the first such block encountered is $P_i=\{p_i\}$ of type (B), with $P_{i+1}=\{p_{i1}, \ldots, p_{i\ell}\}$ ($p_{i1} <\cdots < p_{i\ell}$, $\ell \geq 2$), then modify $(P,c)$ as follows.
\begin{itemize}
\item Replace $P$ with $P'=P_1/\cdots/P_{i-1}/P_i \cup P_{i+1}/P_{i+2}/\cdots/P_k$, and
\item replace $c$ with $c'$ that agrees with $c$ on $P_j$, $j \neq i$, and also satisfies $c'(P_i \cup P_{i+1})=c(P_i)=c(P_{i+1})$ (note that $P_i, P_{i+1}$ have the same color by definition of a type (B) block).
\end{itemize}
(Again, see Example \ref{ex-flip-bij}.)
In this case, set $\iota(P,c)=(P',c')$.

\begin{example} \label{ex-flip-bij}
The pair $(\textcolor{red}{5}/\textcolor{red}{3}/\textcolor{blue}{7}/\textcolor{red}{1}/\textcolor{blue}{246}/\textcolor{blue}{89},c)\in {\mathcal L}_9$ with $c^{-1}({\rm red})=\{1,3,5\}$ and $c^{-1}({\rm blue})=\{2,4,6,7,8,9\}$ gets mapped by $\iota$ to $(\textcolor{red}{5}/\textcolor{red}{3}/\textcolor{blue}{7}/\textcolor{red}{1}/\textcolor{blue}{6}/\textcolor{blue}{24}/\textcolor{blue}{89},c)$, and vice versa.
\end{example}

Call a block a {\it flip block} if it is a block of either type (A) or type (B). If no flip blocks are encountered, then set $\iota(P,c)=(P,c)$.

It is clear that if $(P,c)$ is not a fixed point of $\iota$, then the parity of the number of blocks of $\iota(P,c)$ is different from the parity of the number of blocks of $(P,c)$ ($\iota$ either increases or decreases the number of blocks by $1$). It is also clear, from the definition of ${\mathcal L}_n$ and the construction of $\iota$, that the set of fixed points of $\iota$ (pairs $(P,c)$ without a flip block) is precisely ${\mathcal L_n}$. The proof of \eqref{ltn1-conj'} will thus be completed by showing that $\iota$ is an involution, which we now do.

\begin{claim}
$\iota$ is an involution.
\end{claim}

\begin{proof}
Suppose that the first flip block that is encountered in $(P,c)$ is a block of type (A). With the notation as in the definition of $\iota$ above, consider the pair $(P',c')$. This certainly has a block of type (B), namely $P_i'=p_{i\ell}$. If $P_i'$ is the first block of either type (A) or type (B) in $P'$, then $\iota(P',c')=(P,c)$. We claim that $P_i'$ is indeed the first such block. 

Indeed, if $P_j$ with $j < i-1$ is a flip block in $(P',c')$, then clearly it is also a flip block in $(P',c')$, a contradiction. If $P_{i-1}$ is a block of type (A) in $(P',c')$, then it is a block of type (A) in $(P,c)$, again a contradiction. 

Finally, if $P_{i-1}$ is a block of type (B) in $(P',c')$, then it is a singleton in both $P$ and $P'$, is not the last block in either, has the same color under $c'$ as $P_i'$, so has the same color under $c$ as $P_i$, and its one element is larger than $p_{i\ell}$, so is larger than every element in $P_i$. It follows that $P_{i-1}$ is a block of type (B) in $(P,c)$, a contradiction. The conclusion is that if the first block that is encountered in $(P,c)$ is a block of type (A), then $(\iota \circ \iota)(P,c)=(P,c)$.

Now suppose that the first flip block that is encountered in $(P,c)$ is a block of type (B). With the notation as in the definition of $\iota$ above, consider the pair $(P',c')$. This certainly has a block of type (A), namely $P_i \cup P_{i+1}$. If this is the first flip block in $P'$, then $\iota(P',c')=(P,c)$. We claim that $P_i \cup P_{i+1}$ is indeed the first such block.

Indeed, as before, if either $P_j$ with $j < i-1$ is a flip block in $(P',c')$, or if $P_{i-1}$ is a block of type (A) in $(P',c')$, then we quickly arrive at a contradiction.

The remaining case to consider is when $P_{i-1}$ is a block of type (B) in $(P',c')$. In this case it is a singleton in both $P$ and $P'$, is not the last block in either, it has the same color under $c'$ as $P_i \cup P_{i+1}$, so has the same color under $c$ as $P_i$, and its one element is larger than everything in $P_i \cup P_{i+1}$, so is larger than the one element in $P_i$. It follows that $P_{i-1}$ is a block of type (B) in $(P,c)$, a contradiction. The conclusion is that if the first flip block  that is encountered in $(P,c)$ is a block of type (B), then $(\iota \circ \iota)(P,c)=(P,c)$. This completes the verification that $\iota$ is an involution.
\end{proof}

\subsection{Counting labelled loop-threshold graphs by number of components} \label{sec-lt-comp}

Let $\ltcomp_{n,k}$ be the number of labelled loop-threshold graphs on vertex set $[n]$ with $k$ components. The goal of this section is to give the straightforward verification of the identities (valid for $n, k \geq 1$)
\begin{enumerate}[(i)]
\item $\ltcomp_{n,1} = \sum_{k=0}^{n-1} \eulerian{n}{k}2^k$,
\item for $n \geq 3$ and $2 \leq k \leq n-1$, $\ltcomp_{n,k} = \binom{n}{k-1} \ltcomp_{n-k+1,1}$, and
\item for $n \geq 2$, $\ltcomp_{n,n}=n+1$. 
\end{enumerate}

To see (i), recall that labelled loop-threshold graphs on $[n]$ are encoded by pairs $(s,P)$ where $s\in \{+,-\}$ and $P$ is an ordered partition of $[n]$. From the correspondence given in Section \ref{sec-lt-initial-proofs} we see that if $P$ has an odd number of blocks then $(+,P)$ is connected (its construction ends with the addition of some looped dominating vertices) while $(-,P)$, whose construction ends with the addition of some isolated vertices, is not connected. On the other hand if $P$ has an even number of blocks then $(-,P)$ is connected while $(+,P)$ is not. It follows that half of all loop-threshold graphs on $[n]$ are connected, and (i) follows from \eqref{ltn2-conj'}. (This argument could also have been presented as follows. The map that sends a graph to its complement is an involution on loop-threshold graphs that sends connected graphs to disconnected, and vice-versa).

For (ii): for $n \geq 3$ and $2 \leq k \leq n-1$, the set of labelled loop-threshold graphs on vertex set $[n]$ with $k$ components is obtained by choosing $k-1$ elements from $[n]$, choosing a connected labelled loop-threshold graph on the remaining $n-k+1$ elements, and then adding the chosen $k-1$ elements as isolated vertices. 

For (iii): up to isomorphism there are two loop-threshold graphs on vertex set $[n]$ with $n$ components. One is the edgeless graph, which has one possible labelling. The other consists of a looped vertex together with $n-1$ isolated vertices, and this has $n$  possible labellings.

\section{Proofs for quasi-loop-threshold graph results} 
\label{sec-quasi-loop-thresh-proofs}

Let $\qlt_n$ denote the number of labelled quasi-loop-threshold graphs on $n$ vertices. One goal of this section is to combinatorially establish the formula 
\begin{equation} \label{qltn1'}
\qlt_n = \sum_{k=1}^{n+1} \stirling2{n+1}{k} k^{k-2}.
\end{equation}

The second goal is to show that $\qltconn_n$, the number  of connected labelled quasi-loop-threshold graphs on $n$ vertices, satisfies $\qltconn_2=2$ and for $n \geq 2$  
\begin{equation} \label{qlt-conn'}
\qltconn_n = \sum_{k=1}^n \stirling2{n}{k}
k^{k-1}.
\end{equation}

We start with \eqref{qlt-conn'}, which will prove to be a useful stepping-stone to \eqref{qltn1'}. 
\begin{claim} \label{clm-rpt}
For $n \geq 1$, $\sum_{k=1}^n \stirling2{n}{k}k^{k-1}$ (the right-hand side of \eqref{qlt-conn'}) enumerates labelled quasi-loop-threshold graphs on $n$ vertices that have at least one looped dominating vertex. 
\end{claim}

This immediately yields \eqref{qlt-conn'} --- for $n \geq 2$, being connected and having at least one looped dominating vertex are equivalent, and \eqref{qlt-conn'} is easy for $n=1$.

\medskip

\begin{proof}[Proof (of Claim \ref{clm-rpt})]
Since for $n \geq 1$ the right-hand side of \eqref{qlt-conn'} counts rooted partition trees on $[n]$ --- recall, these are rooted labelled trees whose vertex set is a partition of $[n]$ ---  we will obtain the claim by exhibiting a bijection from labelled quasi-loop-threshold graphs on vertex set $[n]$ that have at least one looped dominating vertex, to rooted partition trees on $[n]$.

The bijection is defined inductively (and is illustrated in Figure \ref{fig-qltconn}). On a vertex set of size one there is just one labelled quasi-loop-threshold graph with a looped dominating vertex, and just one possible rooted tree, so there is nothing to do for $n=1$. 
Given a labelled quasi-loop-threshold graph $G$ on vertex set $[n]$, $n \geq 2$, that has at least one looped dominating vertex, let $A$ be the set of looped dominating vertices. We begin constructing the tree $T'(G)$ associated with $G$ by letting the root of $T'(G)$ have label $A$. 

By the definition of quasi-loop-threshold graphs, on deleting $A$ from $G$ the result is a graph $G'$ that has a number of components, some of which might themselves have looped dominating vertices, and some of which might not (and note that components that do not have a looped dominating vertex must be isolated vertices --- components of size $1$ that are not looped). It is possible that $G'$ may have only one component, and in this case that component must be isolated (if not, then any looped dominating vertex of the sole component of $G'$ would in fact be an element of $A$). From this it follows that if $G'$ has no components that are isolated vertices, than $G'$ must have at least two components (all of which have looped dominating vertices). This is the crucial fact that gets used in the tree construction below, that allows us to distinguish between isolated vertices in $G'$ and components that consist of a single looped vertex.          

Specifically, if $G'$ has some isolated vertices, let $B$ be the set of such vertices, and let the root $A$ of $T'(G)$ have one child whose label is $B$. Each of the remaining components of $G'$, say $C_1, \ldots, C_\ell$, has at least one looped dominating vertex, and so by induction has an associated rooted partition tree. Complete the construction of $T'(G)$ by joining each of $T'(C_1), \ldots, T'(C_\ell)$ to $B$ via edges from  the roots of the $T'(C_i)$. 

If, on the other hand, $G'$ has no isolated vertices, then $G'$ must have at least two components, each of which has a looped dominating vertex. Let $C_1, \ldots, C_\ell$ be these components. Complete the construction of $T'(G)$ by joining each of $T'(C_1), \ldots, T'(C_\ell)$ to $A$ via edges from  the roots of the $T'(C_i)$.   

This process associates a rooted partition tree on $[n]$ to each labelled quasi-loop-threshold graph on vertex set $[n]$ with at least one looped dominating vertex (see Figure \ref{fig-qltconn}), and the process is invertible. This completes the verification that the the right-hand side of \eqref{qlt-conn'} counts labelled quasi-loop-threshold graphs on vertex set $[n]$ with at least one looped dominating vertex. 
\end{proof}

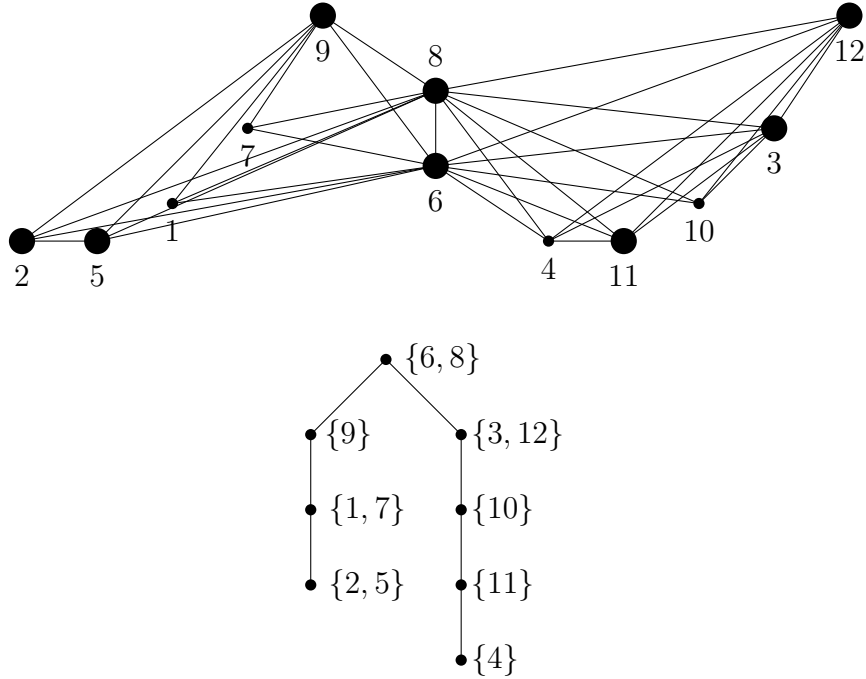
\begin{figure}[ht!] 
\begin{center}
\begin{tikzpicture}
\node[label=below:2] at (0,0) [circle,fill,inner sep=3.5pt]{};
\node[label=below:5] at (1,0) [circle,fill,inner sep=3.5pt]{};
\node[label=below:1] at (2,.5) [circle,fill,inner sep=1.5pt]{};
\node[label=below:7] at (3,1.5) [circle,fill,inner sep=1.5pt]{};
\node[label=below:9] at (4,3) [circle,fill,inner sep=3.5pt]{};
\node[label=below:4] at (7,0) [circle,fill,inner sep=1.5pt]{};
\node[label=below:11] at (8,0) [circle,fill,inner sep=3.5pt]{};
\node[label=below:10] at (9,.5) [circle,fill,inner sep=1.5pt]{};
\node[label=below:3] at (10,1.5) [circle,fill,inner sep=3.5pt]{};
\node[label=below:12] at (11,3) [circle,fill,inner sep=3.5pt]{};
\node[label=above:8] at (5.5,2) [circle,fill,inner sep=3.5pt]{};
\node[label=below:6] at (5.5,1) [circle,fill,inner sep=3.5pt]{};
\draw (5.5,2) -- (5.5,1);
\draw (5.5,2) -- (0,0);
\draw (5.5,2) -- (1,0);
\draw (5.5,2) -- (2,.5);
\draw (5.5,2) -- (3,1.5);
\draw (5.5,2) -- (4,3);
\draw (5.5,2) -- (7,0);
\draw (5.5,2) -- (8,0);
\draw (5.5,2) -- (9,.5);
\draw (5.5,2) -- (10,1.5);
\draw (5.5,2) -- (11,3);
\draw (5.5,1) -- (0,0);
\draw (5.5,1) -- (1,0);
\draw (5.5,1) -- (2,.5);
\draw (5.5,1) -- (3,1.5);
\draw (5.5,1) -- (4,3);
\draw (5.5,1) -- (7,0);
\draw (5.5,1) -- (8,0);
\draw (5.5,1) -- (9,.5);
\draw (5.5,1) -- (10,1.5);
\draw (5.5,1) -- (11,3);
\draw (1,0) -- (0,0);
\draw (4,3) -- (0,0);
\draw (4,3) -- (1,0);
\draw (4,3) -- (2,.5);
\draw (4,3) -- (3,1.5);
\draw (10,1.5) -- (7,0);
\draw (10,1.5) -- (8,0);
\draw (10,1.5) -- (9,.5);
\draw (11,3) -- (7,0);
\draw (11,3) -- (8,0);
\draw (11,3) -- (9,.5);
\draw (11,3) -- (10,1.5);
\draw (7,0) -- (8,0);
\end{tikzpicture}
\end{center}
\begin{center}
\begin{tikzpicture}
\node at (0,0) [circle,fill,inner sep=1.5pt]{};
\node at (.75,0) {$\{6,8\}$};
\draw (0,0) -- (-1,-1);
\draw (0,0) -- (1,-1);
\node at (-1,-1) [circle,fill,inner sep=1.5pt]{};
\node at (-.5,-1) {$\{9\}$};
\node at (1,-1) [circle,fill,inner sep=1.5pt]{};
\node at (1.75,-1) {$\{3,12\}$};
\node at (-1,-2) [circle,fill,inner sep=1.5pt]{};
\node at (-.25,-2) {$\{1,7\}$};
\node at (-1,-3) [circle,fill,inner sep=1.5pt]{};
\node at (-.25,-3) {$\{2,5\}$};
\node at (1,-2) [circle,fill,inner sep=1.5pt]{};
\node at (1.55,-2) {$\{10\}$};
\node at (1,-3) [circle,fill,inner sep=1.5pt]{};
\node at (1.55,-3) {$\{11\}$};
\node at (1,-4) [circle,fill,inner sep=1.5pt]{};
\node at (1.45,-4) {$\{4\}$};
\draw (-1,-1) -- (-1,-3);
\draw (1,-1) -- (1,-4);
\end{tikzpicture}
\end{center}
\caption{A labelled quasi-loop-threshold graph with looped dominating vertex (larger nodes are looped, smaller nodes are unlooped), and its associated rooted partition tree.}
\label{fig-qltconn}
\end{figure}

We now turn to \eqref{qltn1'}. The right-hand side of \eqref{qltn1'} counts partition trees on $[n+1]$ --- labelled (but not rooted) trees whose vertices form a partition of $[n+1]$. Thus our goal now is to describe a bijective correspondence from labelled quasi-threshold graphs on vertex set $[n]$ to partition trees on $[n+1]$.   

So, let a labelled quasi-loop-threshold graph $G$ on vertex set $[n]$ be given. There are three possibilities. 
\begin{enumerate}[(a)]
\item If $G$ has some isolated vertices, let $B$ be the set of isolated vertices, and let $C_1, \ldots, C_\ell$ be the components of $G$ that have looped dominating vertices. We construct an associated tree $T(G)$ as follows. The tree $T(G)$ has a vertex with label $\{n+1\}$. Vertex $\{n+1\}$ has one neighbour, whose label is $B$. Vertex $B$ is also joined to the root of the tree $T'(C_i)$, for each $i$, where $T'(C_i)$ (whose vertices form a partition of the vertex set of $T_i$) is the rooted tree constructed in proof of Claim \ref{clm-rpt}. (See Figure \ref{fig-qltgen-A}.) Note that no special property of the set $[n]$ was used in the proof of Claim \ref{clm-rpt}, so we could have replaced $[n]$ there with an arbitrary finite set.
\begin{figure}[ht!] 
\begin{center}
\begin{tikzpicture}
\node[label=above:1] at (0,0) [circle,fill,inner sep=3.5pt]{};
\node[label=above:5] at (1,0) [circle,fill,inner sep=3.5pt]{};
\node[label=above:2] at (2,0) [circle,fill,inner sep=1.5pt]{};
\node[label=above:3] at (3,0) [circle,fill,inner sep=1.5pt]{};
\node[label=above:4] at (4,0) [circle,fill,inner sep=3.5pt]{};
\draw (0,0) -- (1,0);
\node at (7,0) [circle,fill,inner sep=1.5pt]{};
\node at (6.5,0) {$\{6\}$};
\node at (8,0) [circle,fill,inner sep=1.5pt]{};
\node at (8.75,0) {$\{2,3\}$};
\node at (9,1) [circle,fill,inner sep=1.5pt]{};
\node at (9,1.4) {$\{1,5\}$};
\node at (9,-1) [circle,fill,inner sep=1.5pt]{};
\node at (9,-1.4) {$\{4\}$};
\draw (7,0) -- (8,0);
\draw (8,0) -- (9,1);
\draw (8,0) -- (9,-1);
\end{tikzpicture}
\end{center}
\caption{A labelled quasi-loop-threshold graph of type (a) (left), and its associated partition tree. Here and in Figures \ref{fig-qltgen-B} and \ref{fig-qltgen-C}, larger nodes indicate looped vertices.}
\label{fig-qltgen-A}
\end{figure}
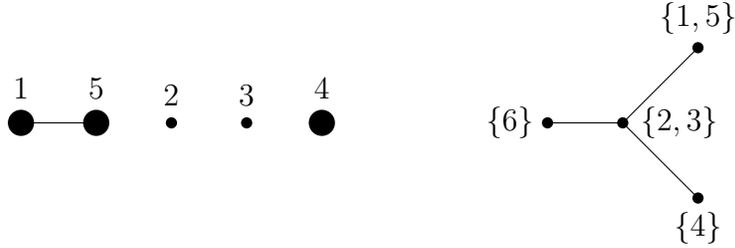
\item If $G$ has no isolated vertices, and also has no looped dominating vertex, then $G$ is the union of components $C_1, \ldots, C_\ell$ ($\ell \geq 2$) each of which has a looped dominating vertices. We construct $T(G)$ as follows. The tree $T(G)$ has a vertex with label $\{n+1\}$. Vertex $\{n+1\}$ is joined to the root of the tree $T'(C_i)$ (from Claim \ref{clm-rpt}), for each $i$. (See Figure \ref{fig-qltgen-B}.)
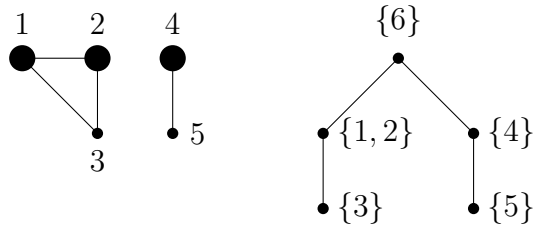
\begin{figure}[ht!] 
\begin{center}
\begin{tikzpicture}
\node[label=above:1] at (0,0) [circle,fill,inner sep=3.5pt]{};
\node[label=above:2] at (1,0) [circle,fill,inner sep=3.5pt]{};
\node[label=below:3] at (1,-1) [circle,fill,inner sep=1.5pt]{};
\node[label=above:4] at (2,0) [circle,fill,inner sep=3.5pt]{};
\node[label=right:5] at (2,-1) [circle,fill,inner sep=1.5pt]{};
\draw (0,0) -- (1,0);
\draw (0,0) -- (1,-1);
\draw (1,0) -- (1,-1);
\draw (2,0) -- (2,-1);
\node at (5,0) [circle,fill,inner sep=1.5pt]{};
\node at (5,.5) {$\{6\}$};
\node at (4,-1) [circle,fill,inner sep=1.5pt]{};
\node at (4.7,-1) {$\{1,2\}$};
\node at (6,-1) [circle,fill,inner sep=1.5pt]{};
\node at (6.5,-1) {$\{4\}$};
\node at (4,-2) [circle,fill,inner sep=1.5pt]{};
\node at (4.5,-2) {$\{3\}$};
\node at (6,-2) [circle,fill,inner sep=1.5pt]{};
\node at (6.5,-2) {$\{5\}$};
\draw (5,0) -- (4,-1);
\draw (5,0) -- (6,-1);
\draw (4,-1) -- (4,-2);
\draw (6,-1) -- (6,-2);
\end{tikzpicture}
\end{center}
\caption{A labelled quasi-loop-threshold graph of type (b) (left), and its associated partition tree.}
\label{fig-qltgen-B}
\end{figure}
\item Finally, if $G$ has some looped dominating vertices, then we construct $T(G)$ as follows. Add $n+1$ to the set that labels the root of $T'(G)$ (from Claim \ref{clm-rpt}), and view the resulting tree (whose vertex set is a partition of $[n+1]$) as an unrooted tree. (See Figure \ref{fig-qltgen-C}.) 
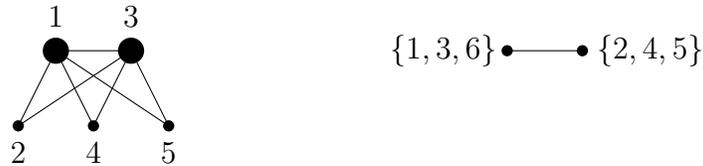
\begin{figure}[ht!] 
\begin{center}
\begin{tikzpicture}
\node[label=above:1] at (0,0) [circle,fill,inner sep=3.5pt]{};
\node[label=above:3] at (1,0) [circle,fill,inner sep=3.5pt]{};
\node[label=below:2] at (-.5,-1) [circle,fill,inner sep=1.5pt]{};
\node[label=below:4] at (.5,-1) [circle,fill,inner sep=1.5pt]{};
\node[label=below:5] at (1.5,-1) [circle,fill,inner sep=1.5pt]{};
\draw (0,0) -- (1,0);
\draw (0,0) -- (-.5,-1);
\draw (0,0) -- (.5,-1);
\draw (0,0) -- (1.5,-1);
\draw (1,0) -- (-.5,-1);
\draw (1,0) -- (.5,-1);
\draw (1,0) -- (1.5,-1);
\node at (6,0) [circle,fill,inner sep=1.5pt]{};
\node at (5.15,0) {$\{1,3,6\}$};
\node at (7,0) [circle,fill,inner sep=1.5pt]{};
\node at (7.9,0) {$\{2,4,5\}$};
\draw (6,0) -- (7,0);
\end{tikzpicture}
\end{center}
\caption{A labelled quasi-loop-threshold graph of type (c) (left), and its associated partition tree.}
\label{fig-qltgen-C}
\end{figure}
\end{enumerate}
This process associates a partition tree on $[n+1]$ to each labelled quasi-loop-threshold graph on vertex set $[n]$, and the association is invertible. This completes the verification of \eqref{qltn1'}.

\section*{OEIS sequences referenced}

This paper concerns the following sequences from OEIS:
\begin{itemize}
\item \seqnum{A000629}, \seqnum{A000670}, \seqnum{A005840}, \seqnum{A008277}, \seqnum{A008292}, \seqnum{A038052}, \seqnum{A048802}, \seqnum{A053525}, \seqnum{A058863}, \seqnum{A058864}, \seqnum{A154921}, \seqnum{A317057}, \seqnum{A348436}, \seqnum{A348576}, \seqnum{A350060}, \seqnum{A350528}, \seqnum{A350531}, \seqnum{A350745} and \seqnum{A350746}.
\end{itemize}

\section*{Keywords and MSC}

{\it 2010 Mathematics Subject Classification}: Primary 05C30. Secondary 05A18, 05A19.

\noindent {\it Keywords}: threshold graph, quasi-threshold graph, loop-threshold graph, quasi-loop-threshold graph, Stirling number of the second kind, Eulerian number, set partition.

\end{document}